\documentclass[10pt]{amsart}
\usepackage{tikz}
\usepackage{tikz-cd}
\usetikzlibrary{matrix, arrows,  decorations.markings,  patterns,  plotmarks}
\usepackage{amsmath}
\usepackage{amssymb}
\usepackage{mathtools}
\usepackage{amsthm}

\usepackage[backend=bibtex, style=alphabetic, maxcitenames=5, maxbibnames=9 ]{biblatex}
\renewbibmacro{in:}{} %This Modifies the In: that otherwise appears in the bib. 
\usepackage{subcaption}

\usepackage{hyperref}
\hypersetup{%
  bookmarksnumbered=true,%
  colorlinks=true,%
  linkcolor=blue,%
  citecolor=blue,%
  filecolor=blue,%
  menucolor=blue,%
  urlcolor=blue,%
  pdfnewwindow=true,%
  pdfstartview=FitBH}
\usepackage{cleveref}

\usepackage{todonotes}

\newcounter{mainthm}
\newtheorem{mainthm}[mainthm]{Theorem}

\newtheorem{thm}{Theorem}[section]
\newtheorem{lem}[thm]{Lemma}

\newtheorem{prop}[thm]{Proposition}
\newtheorem{cor}[thm]{Corollary}

\newtheorem{assumption}[thm]{Assumption}

\theoremstyle{definition}
\newtheorem{defn}[thm]{Definition}
\newtheorem{definition}[thm]{Definition}

\theoremstyle{remark}
\newtheorem{rmk}[thm]{Remark}
\newtheorem{remark}[thm]{Remark}

\newtheorem{example}[thm]{Example}

\crefname{claim}{claim}{claims}
\crefname{prop}{proposition}{propositions}
\crefname{prop}{proposition}{propositions}
\crefname{assumption}{assumption}{assumptions}

\def\C{\mathbb{C}}

\def\N{\mathbb{N}}
\def\P{\mathbb{P}}
\def\Q{\mathbb{Q}}
\def\R{\mathbb{R}}

\def\Z{\mathbb{Z}}

\DeclareMathOperator{\Aff}{Aff}
\DeclareMathOperator{\AffZ}{\Aff_{\Z}}
\DeclareMathOperator{\Alb}{Alb}

\DeclareMathOperator{\Cob}{Cob}
\DeclareMathOperator{\Cobfib}{\mathrm{Cob}_{\mathrm{fib}}}
\DeclareMathOperator{\coker}{coker}
\DeclareMathOperator{\CH}{CH}

\DeclareMathOperator{\GL}{GL}
\DeclareMathOperator{\Hom}{Hom}
\DeclareMathOperator{\id}{id}

\DeclareMathOperator{\Pic}{Pic}

\def\inv{^{-1}}
\newcommand{\mcal}{\mathcal}

\newcommand{\st}{\,\vert\,}
\newcommand{\tx}[1]{\text{#1}}

\newcommand{\vspan}[1]{\langle #1 \rangle}
\newcommand{\wdg}{\wedge}
\def\x{\times}
\newcommand{\xra}[1]{\xrightarrow{#1}}

\title[Fourier transforms and a filtration on the cobordism group of tori]{Fourier transforms and a filtration on the Lagrangian cobordism group of tori}
\author{Álvaro Muñiz-Brea}
\date{}

\begin{document}

\maketitle

\begin{abstract}
    Given a polarized tropical affine torus, we show that the fibered Lagrangian cobordism group of the corresponding symplectic manifold admits a natural geometric filtration of finite length.
    This contrasts with results of Sheridan-Smith in dimension four and the present author in higher dimensions, who showed that such group is infinite-dimensional.

    In the second half of this paper, we construct a Fourier transform between Fukaya categories of dual symplectic tori.  
    We show that, under homological mirror symmetry, it corresponds to the Fourier transform between derived categories of coherent sheaves of dual abelian varieties due to Mukai.
    We use this to show how our filtration is mirror to the Bloch filtration on Chow groups of abelian varieties, but the results may be of broader interest.
\end{abstract}

\tableofcontents

\section{Introduction}
\subsection{Motivation}\label{sec:motivation}
The symplectic topology of a symplectic manifold is often studied through its Lagrangian submanifolds.
It is a long-standing problem to understand what are all the Lagrangians submanifolds of a given symplectic manifold, often up to Hamiltonian isotopy.
As an intermediate step, V. Arnold proposed in the 80s to study Lagrangian submanifolds up to a coarser equivalence relation: Lagrangian cobordism.

\begin{figure}[b]
    \centering
    \begin{tikzpicture}[scale=0.6, thick]

\filldraw[color = black, fill = lightgray]  plot[smooth, tension=.7] coordinates {(-3,0) (-1.5,2) (2,2) (3.5,0) (3.5,-2.5) (0.5,-3.5) (-2.5,-2) (-3,0)};
        \filldraw[color = black, fill = white]  (0,-0.5) ellipse (1 and 0.5);
        
\draw (-8,0) -- (-3,0);
    
\draw (3.5,0) -- (8.5,0);
        
\draw[dashed]  (-3.5,4) rectangle (4.5,-4.5);
        
\node[scale = 1] at (-5.5,0.5) {$L_- \times \mathbb{R}_{<a_-}$};
        \node[scale = 1]  at (7,0.5) {$L_+  \times \mathbb{R}_{>a_+}$};
        \node[scale = 1]  at (-1.5,0.5) {$\pi_{\mathbb{C}}(V)$};
        \node[scale = 1] at (5,3) {$K$};
        
    \end{tikzpicture}
    \caption{Lagrangian cobordism between $L_-$ and $L_+$. The shaded region (including the two horizontal lines) depicts a typical projection of a two-ended Lagrangian cobordism $V \subset X \x \C$ to $\C$. We have included (in dashed lines) an example of a compact region $K \subset\C$ outside which $V$ is product type: in $\C \setminus K$ the projection looks like two straight lines, and living over them in $X \x \C$ we have the cylidrical Lagrangians $L_- \x \R_{<-a_-}$ and $L_+ \x \R_{>a_+}$.}
    \label{fig:lagcob}
\end{figure}
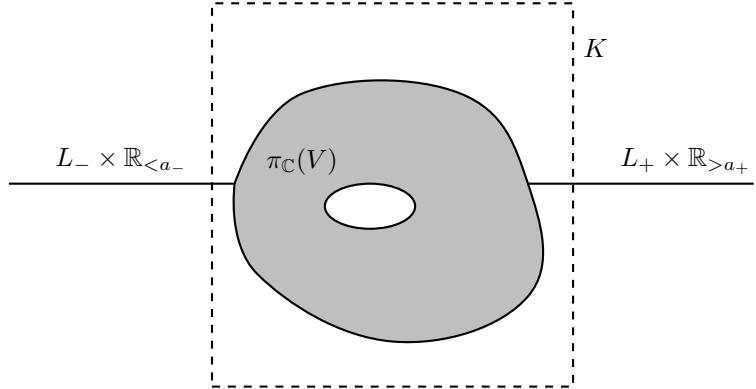

\begin{defn}[\cite{arnol1980lagrange,arnol1980lagrange2}]\label{defn:lagrangian_cobordism}
    Let $X$ be a symplectic manifold.
    Lagrangian submanifolds $L_\pm \subset X$ are said to be \emph{Lagrangian cobordant} if there exists a properly embedded Lagrangian submanifold $V \subset (X \x \C, \omega_X \oplus \omega_\C)$ such that
    \begin{equation*}
        V  = L_- \x \R_- \cup L_+ \x \R_+
    \end{equation*}
    outside (the preimage of) some compact set $K \subset \C$ (see \Cref{fig:lagcob}).

    Similarly, we say tuples $(L_0^\pm,\dots, L_{k_\pm}^\pm)$ are Lagrangian cobordant if there is a cobordism fibering over rays $\R_\pm + i m_\pm$, $0\leq m_\pm \leq k_\pm$ outside the preimage of a compact set.
\end{defn}

Lagrangian cobordism is an equivalence relation between Lagrangian submanifolds which is strictly coarser than Hamiltonian isotopy.
One obtains a linear-algebraic invariant of a symplectic manifold by considering the free group on Lagrangians modulo this relation:

\begin{defn}\label{def:cobsimple}
    The \emph{Lagrangian cobordism group} of $X$ is the quotient
    \begin{equation*}
        \Cob(X) := \Z\vspan{L \subset X} / \sim
    \end{equation*}
    where $\sim$ is the equivalence relation generated by expressions of the form
    \begin{equation*}
        L_0^- + \dots + L_{k_-}^- = L_0^+ + \dots + L_{k_+}^+
    \end{equation*}
    if there is a Lagrangian cobordism between the tuples $(L_0^\pm,\dots, L_{k_\pm}^\pm)$.
\end{defn}

\begin{remark}\label{rmk:decorations_cob}
    One often restricts the generators $L\subset X$ to Lagrangians satisfying some additional properties (such as being monotone) and/or carrying some extra data (such as a grading).
    In such case the Lagrangian cobordisms giving the relations in the cobordism group are required to also carry the extra data, and the restriction must be compatible. 
    We explain this in more detail in \Cref{sec:lagcobgroup}.
\end{remark}

For compact symplectic manifolds, Lagrangian cobordism groups are not easy to compute once the dimension  of $X$ is greater than two.\footnote{A different story holds for non-compact symplectic manifolds and their \emph{exact} Lagrangian submanifolds. With a suitable definition of a Lagrangian cobordism in such setting, Bosshard showed in \cite{bosshard2023note} that Lagrangian cobordism groups of Weinstein manifolds reduce to (relative) middle homology.}
This was made clear by Sheridan-Smith in \cite{sheridan2020rational}, who showed that any symplectic $4$-manifold with vanishing first Chern class has an infinite-dimensional (in a suitable sense) Lagrangian cobordism group.
Their result holds in fact for any symplectic manifold with vanishing first Chern class and dimension greater than two, as proved by the present author in \cite[Proposition 1.6]{muniz2024lagrangian}.

In this paper we elucidate some structure on the infinite-dimensional cobordism group of symplectic tori.

  \subsection{Main result}
Let $B$ be a tropical affine torus. 
This means $B$ is a smooth torus $B \simeq T^n$ equipped with an integrable lattice $T^*_\Z B \subset T^*B$ of covectors (see \Cref{sec:tropicalgeometry} for more details).
We consider symplectic manifolds of the type
\begin{equation}\label{eq:BgivesX(B)}
    X = X(B) := T^* B / T^*_\Z B.
\end{equation}
These come equipped with a Lagrangian torus fibration $\pi: X \to B$.
They are topologically tori, namely
\begin{equation}
    X \cong B \x T^*_0 B / T^*_{\Z,0} B \cong B \x T^n
\end{equation}
and their symplectic structure is that induced from $T^*B$, i.e. $\omega_X = \sum_i dq_i \wedge dp_i$ for $(q_i)$ coordinates on $B$ and $(p_i)$ the dual coordinates.

Among tropical affine tori we will be interested in those that are \emph{polarized}.
The precise definition is given in \Cref{def:polarized_torus}. 
For now, let us just say that this ensures that the algebro-geometric mirror to $X$ will be projective. 
Symplectically, it gives us a tropical description of constant Lagrangian sections (cf. \Cref{sec:constantsections}).

The base $B$ of the Lagrangian torus fibration $\pi: X \to B$ has a group structure induced from that of $\R^n$ (any tropical affine torus $B$ is the quotient of $\R^n$ by a lattice).
Let
\begin{equation*}
    m: B \x B \to B
\end{equation*}
denote the group operation on $B$.
\begin{definition}\label{def:pontryagin_intro}
    Let $\Cobfib(X) \subset \Cob(X)$ be the subgroup generated by Lagrangian fibers of $\pi: X \to B$.\footnote{The precise cobordism group that we consider (see \Cref{rmk:decorations_cob}) is explained in \Cref{sec:lagcobgroup}.}
    The \emph{Pontryagin product} on $\Cobfib(X)$ is the homomorphism
    \begin{align*}
        \star: \Cobfib(X) \otimes \Cobfib(X) &\to \Cobfib(X) \\
        F_b\star F_{b'} &= F_{m(b,b')}.
    \end{align*}
    That such map is well-defined is proved in \Cref{lem:pontryagin_on_cob}.
\end{definition}

The Pontryagin product turns $\Cobfib(X)$ into a ring.
The subgroup 
$$
\Cobfib(X)_{\hom} \subset \Cobfib(X)
$$
generated by elements of the form $F_p - F_q$ is easily seen to be an ideal for this ring structure.
We define the \emph{parallelotope filtration} on $\Cobfib(X)$ as the decreasing filtration generated by $\Cobfib(X)_{\hom}$, that is:
$$
    F^i\Cobfib(X) := \Cobfib(X)_{\hom}^{\star i}
$$
where we set $F^0\Cobfib(X) := \Cobfib(X)$.
In this paper we prove:
\begin{mainthm}[= \Cref{thm:filtration_terminates}]\label{mainthm:filtration_terminates}
    Let $B$ be a polarized tropical affine torus of dimension $n$ and $X = X(B)$ the corresponding symplectic manifold.
    Then the parallelotope filtration on $\Cobfib(X)$ satisfies $F^{n+1}\Cobfib(X) = 0$.
\end{mainthm}

We sketch the proof of \Cref{mainthm:filtration_terminates} in \Cref{sec:proof_sketch}.
The significance of \Cref{mainthm:filtration_terminates} is that, even though the (fibered) Lagrangian cobordism group of tori is infinite-dimensional \cite{sheridan2020rational,muniz2024lagrangian} and thus, a priori, intractable, when $B$ is polarized one can still say something about it.
Namely, there exists a geometric filtration, whose first graded pieces 
\begin{align*}
    F^0 \Cobfib(X)/F^1 \Cobfib(X) &=  \Cobfib(X) / \Cobfib(X)_{\hom} \cong H_0(B)\\
    F^1 \Cobfib(X)/F^2 \Cobfib(X) &=  \Cobfib(X)_{\hom} / \Cobfib(X)_{\hom}^{\star 2} \cong \Alb(B)
\end{align*}
are understood, and which terminates in a \emph{finite} (in fact, $n+1$) number of steps (we prove these two isomorphisms in \Cref{lem:F0F1} and \Cref{lem:F1F2}).
The group $\Alb(B)$ is the \emph{Albanese variety of $B$}; for the purposes of this paper, it suffices to say that it is an abelian group built from the tropical geometry of the base $B$, and it is well understood.

\begin{remark}
    The assumption that $B$ is polarized cannot be dropped from \Cref{mainthm:filtration_terminates}.
    In fact, for non-polarizable $B$ not only the parallelotope filtration need not have length at most $n+1$, but in fact can be \emph{unbounded}.
    This follows from a result of Sheridan-Smith, who showed that tuples of fibers in a very general (in particular, non-polarized) tropical affine torus  are Lagrangian cobordant if and only if they coincide as unordered tuples \cite[Theorem 1.3]{sheridan2021lagrangian}.
    Note that this implies that $F^k\Cobfib(X) \neq 0$ for \emph{all} $k$.
    Indeed, clearly $F^1\Cobfib(X)\neq 0$ since any two distinct points in $B$ yield non-cobordant fibers.
    Assuming $F^k\Cobfib(X) \neq 0$, let $0\neq z:= \star_{i=1}^k(F_{b_i}- F_{b_i'}) \in F^k\Cobfib(X)$.
    By choosing any $b\in B$ sufficiently close to the origin (or, more generally, any $b$ such that $m(b,b_i)\neq b_j$ for all $i,j$), the classes $F_b \star z$ and  $F_0\star z = z$ consist of distinct ordered tuples, and thus cannot be cobordant. 
    It follows that $0 \neq (F_b - F_0)\star z \in F^{k+1}\Cobfib(X)$.
\end{remark}

 \subsection{Fourier transforms}\label{sec:introfourier}
In this section we present a construction of a \emph{Fourier transform}
\begin{equation*}
    \mathfrak F: \mcal Fuk(X(B)) \to \mcal Fuk(X(B^\vee))
\end{equation*}
for dual tropical affine tori $B,B^\vee$ (the dual of a tropical affine torus is defined in \Cref{def:dualtropicalaffinetorus}).
\begin{remark}
    While $\mathfrak F$ might be of independent interest and puts \Cref{mainthm:filtration_terminates} in a broader context (namely, relating it to classical results in algebraic geometry), it is \emph{not} needed to prove \Cref{mainthm:filtration_terminates}.
\end{remark}

The motivation for $\mathfrak F$ comes from algebraic geometry.
Given an abelian variety $A$, there exists a \emph{dual abelian variety} $A^\vee$ which is defined as the moduli-space of degree-zero line bundles on $A$.
The product $A \x A^\vee$ admits a universal line bundle $\mcal P \to A \x A^\vee$---the `Poincaré bundle'---which in turn induces a Fourier-Mukai transform \cite{huybrechts2006fourier}
$$
    \mcal F:= \Psi_{\mcal P}: D^bCoh(A) \xra{\sim} D^bCoh(A^\vee).
$$

One can get abelian varieties from polarized tropical affine tori as follows.
Starting from a tropical affine torus $B$, one considers the rigid-analytic space $Y(B)$ over the Novikov field as defined by Kontsevich-Soibelman \cite{kontsevich2006affine}.
In general, this is a rigid-analytic torus---but when $B$ is polarized, $Y(B)$ is the analytification of an abelian variety, which we denote still by $Y(B)$.
Furthermore, we show in \Cref{lem:SYZ_dual} that the dual abelian variety $Y(B)^\vee$ can be identified with $Y(B^\vee)$, where $B^\vee$ is the dual tropical affine torus to $B$.

Now recall that tropical affine tori also give rise to symplectic manifolds $X(B)$ via \Cref{eq:BgivesX(B)}.
It follows from a (more general) result of Abouzaid \cite{abouzaid2014family,abouzaid2017family,abouzaid2021homological} that there is a fully faithful embedding
$$
 \Phi_B: \mcal Fuk(X(B)) \to D^bCoh(Y(B)),
$$
and it is an equivalence when $B$ is polarized.
This yields the following natural question: `what is the mirror functor to $\mathcal F$?'.

\begin{definition}
    Let $B$ be a tropical affine torus and $B^\vee$ its dual.
    Then we define the \emph{(symplectic) Fourier transform} as the equivalence $\mathfrak{F}$ that makes the following diagram commute:
    $$
    \begin{tikzcd}
        \mcal Fuk(X(B)) \arrow[r,"\mathfrak F"] \arrow[d,"\sim","\Phi_B"'] & \mcal Fuk(X(B^\vee)) \arrow[d,"\sim","\Phi_{B^\vee}"']\\
        D^bCoh(Y(B)) \arrow[r,"\sim","\mcal F"'] & D^bCoh(Y(B)^\vee) 
    \end{tikzcd}
    $$
\end{definition}

Let us characterize $\mathfrak F$.
Consider the symplectomorphism\footnote{To be precise, $\iota$ is a \emph{graded} symplectomorphism, the grading data being that of \Cref{eq:gradediota}.}
\begin{align*}
    \iota: X(B) &\xra{\sim} X(B^\vee) \\
    (q,p) &\mapsto (-p,q).
\end{align*}
Here we  are implicitly using the natural identification of $B^\vee$ with the fibre of $X(B) \to B$, and B with the fibre of $X(B^\vee) \to B^\vee$ (see \Cref{rmk:dualtorusisfiber}).
It induces an equivalence  $\varphi_\iota:\mcal Fuk(X(B)) \to \mcal Fuk(X(B^\vee))$, and in this paper we show:
\begin{mainthm}[= \Cref{prop:symplecticfourier}]\label{mainthm:symplecticfourier}
    Let $\Phi_B : \mcal Fuk(X(B)) \to D^bCoh(Y(B))$ be the homological mirror symmetry equivalence due to Abouzaid \cite{abouzaid2021homological}.
    Then the diagram
    $$
    \begin{tikzcd}
        \mcal Fuk(X(B)) \arrow[r,"\varphi_\iota"] \arrow[d,"\sim","\Phi_B"'] & \mcal Fuk(X(B^\vee)) \arrow[d,"\sim","\Phi_{B^\vee}"']\\
        D^bCoh(Y(B)) \arrow[r,"\sim","\mcal F"'] & D^bCoh(Y(B)^\vee) 
    \end{tikzcd}
    $$
    commutes.
    In other words, the symplectic Fourier transform is the functor induced by the symplectomorphism $\iota$.
\end{mainthm}

\begin{remark}
    While the algebraic Fourier transform $\mcal F$ is \emph{not} induced by an isomorphism $A \to A^\vee$ (if it were, skyscraper sheaves of points would map to skyscraper sheaves of points, rather than line bundles), \Cref{mainthm:symplecticfourier} shows that its natural mirror $\mathfrak F$ is induced by a (graded) symplectomorphism $X(B) \to X(B^\vee)$.
\end{remark}

\subsection{Connections between \Cref{mainthm:filtration_terminates} and Chow groups}\label{sec:connections_ag}
The reader familiar with Chow groups in algebraic geometry will have noticed a similarity between \Cref{mainthm:filtration_terminates} and a classical result of Bloch for the Chow group of an abelian variety \cite{bloch1976some}. 
In this section we make this connection explicit.

Let us start by recalling the definition of Chow groups.
One way to understand the algebraic geometry of an algebraic variety is through a study of its algebraic subvarieties.
The collection of all subvarieties is too large to be studied, so one often imposes some equivalence relation among them.
Two $k$-dimensional subvarieties $Z_\pm \subset Y$ of an algebraic variety $Y$ are said to be \emph{rationally equivalent} if there exists a $(k+1)$-dimensional subvariety $W \subset Y \x \P^1$ dominant over $\P^1$  such that
$$
W \cap (Y \x \{\pm\infty\}) = Z_\pm.
$$
Let $Z_k(Y)$ denote the free abelian group generated by $k$-dimensional subvarieties of $Y$ and $R_k(Y) \subset Z_k(Y)$ the subgroup generated by differences $Z_+ - Z_-$ of rationally equivalent subvarieties.
The \emph{Chow groups} of $Y$ are the quotients
$$
\CH_k(Y) := Z_k(Y) / R_k(Y).
$$

Chow groups enjoy many nice properties, among them functoriality: for any morphism $f: Y_1 \to Y_2$ between algebraic varieties, one gets an induced morphism 
$$
f_*: \CH_k(Y_1) \to \CH_k(Y_2).
$$
The action on this morphism on a subvariety is either taking its image if it has the right dimension, or zero if the image has smaller dimension (we refer the reader to \cite{fulton2013intersection} for more details on Chow groups).

Now let $Y = A$ be an abelian variety. 
The group structure $\mu : A \x A \to A$ induces by functoriality a homomorphism between Chow groups, which composed with the Künneth map $\CH_i(A)\otimes \CH_j(A) \to \CH_{i+j}(A \x A)$ 
yields a homomorphism 
\begin{equation}\label{eq:CH_pontryagin}
    \star : \CH_i(A) \otimes \CH_j(A) \to \CH_{i+j}(A)
\end{equation}
which is known as the \emph{Pontryagin product}.
We will be interested in the subgroup $\CH_0(A)$; note this is a subring for the Pontryagin product.
The subgroup $\CH_0(A)_{\hom} \subset \CH_0(A)$ of $0$-cycles of degree zero is an ideal for the Pontryagin product, and we can define the Bloch filtration on $\CH_0(A)$ as the decreasing filtration generated by $\CH_0(A)_{\hom}$:
$$
F^i\CH_0(A) := \CH_0(A)_{\hom}^{\star i}.
$$
Bloch showed:
\begin{thm}[\cite{bloch1976some}]\label{thm:bloch}
    Let $A$ be an abelian variety of dimension $n$ over an algebraically closed field.
    Then the Bloch filtration on $\CH_0(A)$ satisfies $F^{n+1}\CH_0(A) = 0$.
\end{thm}

The connection between \Cref{mainthm:filtration_terminates} and \Cref{thm:bloch} is made apparent using homological mirror symmetry.
Recall from \Cref{sec:introfourier} that, using Abouzaid's homological mirror symmetry result \cite{abouzaid2021homological}, one can construct a fully faithful embedding 
$$
\Phi_B : \mcal Fuk(X(B)) \to D^bCoh(Y(B))
$$
for a tropical affine torus $B$.
The space $Y(B)$ is a rigid-analytic torus; and when $B$ is polarized, $Y(B)$ is the analytification of an abelian variety $A$ and $\Phi_B$ is an equivalence.
A rigid-analytic GAGA principle says that $D^bCoh(Y) = D^bCoh(A)$ (the former being built from analytic coherent sheaves, the latter from algebraic coherent sheaves), so that putting everything together one has an equivalence $\mcal Fuk(X) \simeq D^bCoh(A)$.

Now, in favorable situations, work of Biran-Cornea \cite{biran2014lagrangian,biran2021lagrangian} guarantees the existence of a map 
$$
    \Cob(X) \to K_0(\mcal Fuk (X))
$$
where $K_0(-)$ denotes the Grothendieck group of a triangulated category.
We can thus consider the composite map 
$$
\begin{tikzcd}
    K_0(\mcal Fuk(X)) \arrow[r,"\sim"] & K_0(D^bCoh(A)) \arrow[d,"ch"] \\
    \Cob(X) \arrow[u] \arrow[r,dashed,"\Psi"] & \CH_*(A)_\Q
\end{tikzcd}
$$
where $ch$ is the Chern character.
It restricts to a map $\Cobfib(X) \to \CH_0(A)$, and we show in \Cref{prop:Cob_to_CH_ringhom} that it is a ring homomorphism (where both groups are equipped with their corresponding Pontryagin products).

\begin{remark}
    The homomorphism $\Cobfib(X) \to \CH_0(A)$ is \emph{not} known to be injective, so we would not be able to conclude that the parallelotope filtration on $\Cobfib(X)$ terminates from the analogous result for Chow groups.
    Furthermore, in the present setting the construction of a map $\Cob(X) \to K_0(\mcal Fuk(X))$ has not yet appeared in the literature.
\end{remark}

Lastly, we note that Bloch does not use any kind of Fourier transform to prove his result.
Instead, Beauville is the one that reformulates Bloch's result in terms of Fourier-Mukai transforms in \cite{beauville2006quelques}.
It is their approach that motivated the techniques developed in this paper.

\subsection{Idea of the proof of \Cref{mainthm:filtration_terminates}}\label{sec:proof_sketch}
The proof of \Cref{mainthm:filtration_terminates} is based on two main ideas, both geometric.

The first is the realisation that one can exchange questions about Pontryagin products of fibers to questions about fiberwise addition of sections (of a different Lagrangian torus fibration).
Indeed, note that $X(B) \to B$ has an additional natural Lagrangian torus fibration given by projecting to the fibers (since $X(B) \simeq B \x F$ is trivial), and that \emph{fibers of $X(B) \to B$ are flat sections of $X(B) \to F$}.
We can summarize this in the following diagram:
$$
\begin{tikzcd}
    & F \ar[d, hookrightarrow] &  \\
    B \ar[r, hookrightarrow] & X(B) \ar[d, "\pi"'] \ar[r, "\pi'"] & F \\
    & B & .
\end{tikzcd}
$$
Furthermore, the Pontryagin product on $\Cobfib(X)$ is defined by using the group structure of the base $B$ (see \Cref{def:pontryagin_intro}). 
On the other hand, also by construction (see \Cref{sec:fiberwiseaddition}), the fiberwise addition operation on $X(B)$ coming from the fibration $X(B) \to F$ is defined using the group structure on its fibers---which are precisely $B$.
That is, emphasizing with a subscript the fibration used to define each operation, we have that 
\begin{equation}\label{eq:pontryagin_vs_fiberwise}
    F_1 \star_{\pi} F_2 = F_1 \otimes_{\pi'} F_2
\end{equation}
for any two fibers $F_i$ of $X(B) \to B$.
One can thus translate the iterated fiberwise addtion of \Cref{mainthm:filtration_terminates} to a question about fiberwise addition of flat sections of $X(B) \to F$.
\begin{remark}
    Up to now we have only used the fact that $B$ is a tropical affine torus, and not that it is polarized.
\end{remark}

The second idea is to study these flat sections of $\pi'$ using the tropical geometry of the base $F$.
We show in \Cref{cor:constantsectionistropical} that when $B$ is polarized flat sections are Hamiltonian isotopic to graphs of tropical rational functions. 
It is here that polarizability is essential---we illustrate this in \Cref{ex:polarization_2_torus}.
Using Hicks' surgery cobordism \cite{hicks2020tropical} we can turn differences $\Gamma(df_1) - \Gamma(df_2)$ of flat sections (recall these correspond to the generators $F_b - F_{b'}$ of $\Cobfib(X)_{\hom}$) into Lagrangians lifting tropical varieties of codimension $\geq 1$ (\Cref{cor:difference_of_flat_sections_codimension}), so that our original problem has now been reduced to studying fiberwise additions of tropical Lagrangians.
To finish, note that fiberwise addition of $n+1$ Lagrangians lifting \emph{transverse} tropical subvarieties of codimension $\geq 1$ is necessarily empty.
We show in \Cref{lem:transversality} that we can arrange our tropical subvarieties to be transverse, thus completing the proof.

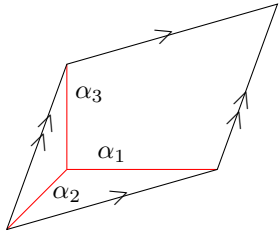
\begin{figure}
    \centering
    \begin{tikzpicture}
    \draw[red] (-0.8, -0.8) -- (0,0);
    \draw[red] (0, 1.4) -- (0,0);
    \draw[red] (2, 0) -- (0,0);
    
    \draw (-0.8, -0.8) -- (0, 1.4) -- (2.8, 2.2) -- (2,0) -- (-0.8, -0.8);

    \node[rotate = 70] at (-0.3,0.5) {$>>$};
    \node[rotate = 70] at (2.3,0.8) {$>>$};

    \node[rotate=20] at (1.3,1.8) {$>$};
    \node[rotate=20] at (0.7,-0.35) {$>$};
    \node at (0.6,0.2) {$\alpha_1$};
    \node at (0.3,1) {$\alpha_3$};
    \node at (0,-0.33) {$\alpha_2$};
    \end{tikzpicture}
    \caption{Family of polarized tropical affine tori parametrized by $(\alpha_1,\alpha_2,\alpha_3) \in \R^3_{>0}$. The labels $\alpha_i$ indicate the length of each edge, which have directions $(1,0),(-1,-1)$ and $(0,1)$ respectively.}
    \label{fig:polarised_torus}
\end{figure}

\begin{example}\label{ex:polarization_2_torus}
    To see the importance of a polarization on $B$ we need to consider $\dim B = 2$ (every $1$-dimesional torus is polarizable, so it would not illustrate the point).
    
    We consider the family of tropical tori
    $$
        B = B_{\alpha} = \R^2 / \Lambda_{\alpha}, \quad \Lambda_{\alpha} = \Z \vspan{(\alpha_1+\alpha_2, \alpha_2), (\alpha_2, \alpha_2+\alpha_3)}
    $$
    for $\alpha = (\alpha_1,\alpha_2,\alpha_3) \in \R^3_{>0}$ (see \Cref{fig:polarised_torus}).
    This is a $3$-dimensional family in the $4$-dimensional space of tropical affine tori, all of which are polarized (\Cref{ex:polarization_on_torus}).
    Then, denoting by $t_b = m(b,-): B \to B$ the translation by $b \in B$ and by $\phi\in H^0(C^\infty/\Aff)$ an strictly convex function (for instance, the square of the norm defined by the metric $g$ in \Cref{eq:metricfrompolarisation}), we can use the tropical sections
    $$
    \Gamma(d\Phi_{b}), \quad \Phi_b := \phi - t_b^*\phi
    $$
    to parametrise flat Lagrangian sections of $X(B) \to B$.\footnote{We explain in \Cref{sec:lagcobgroup} how to make sense of the graph of a tropical polynomial (which is not a smooth function and hence does not have, a priori, a well-defined differential).}
    The reason is essentially the following: one can easily check that the Lagrangian sections $\Gamma(d\Phi_b)$ are Hamiltonian isotopic to flat sections, and furthermore each of them has a different Hamiltonian isotopy class, so they parametrize the whole space of flat sections.
    A rigorous proof of this can be found in \Cref{sec:constantsections}.

    Now suppose instead that $B$ is a sufficiently irrational torus that does not contain any tropical curve, i.e. to which \cite[Theorem 1.3]{sheridan2021lagrangian} applies (explicit examples are given in that paper).
    If we consider the dual torus $B^\vee$ (this is defined precisely in \Cref{def:dualtropicalaffinetorus}, but for now one can think of $B^\vee\simeq F$ as the fiber $F\subset X(B) \to B$), then \Cref{mainthm:filtration_terminates} cannot apply to $B^\vee$, at least for the class of cobordism considered in the statement of \cite[Theorem 1.3]{sheridan2021lagrangian}.
    Indeed, \Cref{mainthm:filtration_terminates} implies the existence of non-trivial cobordisms between flat sections of $X(B) \to B$, which can be thought of as cobordisms between fibers of $X(B^\vee) \to B^\vee$. 
    By contrast,  \cite[Theorem 1.3]{sheridan2021lagrangian} says that there are no non-trivial cobordism between fibers of $X(B^\vee) \to B^\vee$.
\end{example}

 \paragraph{{\bf Acknowledgements}}
I am very grateful to my supervisor Nick Sheridan for many useful discussions and ideas that helped me to complete this work.
He has also been very generous with his time, patient and supportive.
I am also thankful to Jeff Hicks for developing and teaching me about Lagrangian surgery, a technique that was crucial for the proofs I present here, as well as acting like a second advisor in many aspects of my research.
The author was supported by an ERC Starting Grant (award number 850713-HMS).

\section{Background}
\subsection{Tropical geometry}\label{sec:tropicalgeometry}
Let $\Aff(\R^n) = \R^n \rtimes \GL(n,\Z)$ denote the group of integral affine transformations of $\R^n$.
Such transformations are of the form
$$
x \in \R^n \mapsto Ax + b \in \R^n
$$
for $A \in \GL(n,\Z)$ and $b \in \R^n$.
\begin{defn}
    A \emph{tropical affine manifold} is a smooth manifold $B$ together with an atlas of \emph{affine charts}, i.e. charts $\{\varphi_i: U_i \to \R^n\}$ such that the transition functions $\varphi_i\circ\varphi_j\inv$ are in $\Aff(\R^n)$.
\end{defn}
\begin{example}\label{ex:tropicalaffinetorus}
    Consider the smooth manifold $\R^n$ with the standard lattice subbundle 
    $$
    T_\Z\R^n := \R^n \x \Z^n \subset \R^n \x \R^n \cong  T\R^n.
    $$
    Given any other $n$-lattice $\Lambda \subset \R^n$, the quotient
    $$
    B := \R^n / \Lambda,
    $$
    where $\Lambda$ acts on $\R^n$ by translation, inherits an affine structure by projecting $T_\Z \R^n$.
    The tropical affine manifold $(B, T_\Z B)$ is called a \emph{tropical affine torus}.
\end{example}

A tropical affine manifold comes equipped with the following important sheaves:
\begin{itemize}
    \item $\Aff \subset C^\infty$ is the \emph{sheaf of affine functions on $B$}. Local sections of $\Aff$ are smooth functions $f:U \to \R$ that in a local affine chart of $B$ are of the form 
    $$
    f(x_1,\dots,x_n) = \sum a_i x_i + b
    $$
    for some $a_i \in \Z$ and $b \in \R$.

    \item $T^*_\Z B \subset T^*B$ is a locally constant sheaf of $\Z$-modules defined as the image of $\Aff$ under the differential $d:C^\infty \to T^*B$.
    
    \item $T_\Z B \subset TB$ is a locally constant sheaf of $\Z$-modules that is dual to $T^*_\Z B$: a vector $v \in TB$ is in $T_\Z B$ if and only if $\alpha(v) \in \Z$ for all $\alpha \in T^*_\Z B$.
\end{itemize}
These sheaves fit into short exact sequences
\begin{equation}\label{eq:sesaffcinfty}
    0 \to \Aff \to C^\infty \to C^\infty / \Aff \to 0
\end{equation}
\begin{equation}\label{eq:sesRAff}
0 \to \R \to \Aff \xra{d} T^*_\Z B \to 0.
\end{equation}
In fact, \Cref{eq:sesRAff} gives another definition of $T^*_\Z B$.

We will also be interested in \emph{tropical subvarieties} of $B$.
These are weighted rational polyhedral complexes in $B$ of pure dimension satisfying the \emph{balancing condition}: if $V$ is a weighted rational polyhedral complex of pure dimension $k$, then for every $(k-1)$-dimensional face $\sigma$ of $V$ one has 
$$
\sum_{\tau \supset \sigma} w_\tau v_{\tau} = 0 \in \R^n / T \sigma
$$
where the sum runs over all top-dimensional faces $\tau$ that contain $\sigma$, $w_\tau$ is the weight of $\tau$, and $v_\tau$ is a non-zero vector tangent to $\tau$ and normal to $\sigma$; see \cite{mikhalkin2006tropical} for a precise definition.

\subsection{Tropical Lagrangian sections}\label{sec:tropicallagrangiansections}
Tropical geometry is intimately related to symplectic geometry via the following construction.
Let $B$ be a tropical affine manifold.
The natural symplectic form on the cotangent bundle $T^*B$ is invariant under translation, hence the quotient manifold
$$
X(B) := T^*B / T^*_\Z B
$$
inherits a symplectic form.
The projection 
$$
\pi: X(B) \to B
$$ 
is a Lagrangian torus fibration.

One can construct Lagrangian sections of $\pi$ using the tropical geometry of the base $B$ as follows.
For the case of the cotangent bundle $T^*B \to B$, any smooth function $f \in H^0(C^\infty)$ defines a Lagrangian section $\Gamma(df) \subset T^*B$.
Now note that, by definition, for $f \in \Aff$ we have $\Gamma(df) \in T^*_\Z B$, hence $\Gamma(df) = \Gamma_0$ is just the zero-section of $X(B)$.
It follows that, given $f \in H^0(C^\infty/\Aff)$, (the graph of) its differential is well-defined as a section of $T^*B/T^*_\Z B$.
Conversely, every Lagrangian section is locally the graph of a smooth function, and such smooth function is unique up to addition of an affine function.
In other words, the group $H^0(C^\infty/\Aff)$ parametrises Lagrangian sections of $\pi$.

Global sections of $C^\infty$ give rise to Lagrangians in $X(B)$ which lift to $T^*B$.
These Lagrangians are graphs of exact $1$-forms, hence Hamiltonian isotopic to the zero-section. 
Note there is a natural Hamiltonian isotopy that decends to $X(B)$.
It follows that the quotient group
$$
H^0(C^\infty/\Aff) / H^0(C^\infty)
$$
parametrises Lagrangians sections of $\pi$ up to Hamiltonian isotopy.

Using the short exact sequence (\ref{eq:sesaffcinfty}) and the fact that $H^1(C^\infty)=0$ (since the sheaf $C^\infty$ is soft), we have that
\begin{equation}\label{eq:H1Affasquotient}
    H^0(C^\infty/\Aff) / H^0(C^\infty) \cong H^1(\Aff).
\end{equation}
The latter group---a purely tropical object---parametrises Lagrangian sections of $\pi$ up to Hamiltonian isotopy. \subsection{Lagrangian cobordisms}\label{sec:lagcobgroup}
We recall from \Cref{defn:lagrangian_cobordism} that a \emph{Lagrangian cobordism} between tuples of Lagrangians $(L_0^\pm,\dots,L_{k^\pm}^\pm)$ is a properly embedded Lagrangian submanifold $V \subset X \x \C$ such that
$$
V = \bigsqcup_{j=0}^{k^-} \left(L^-_j \x (\R_- + i \cdot j)\right) \sqcup \bigsqcup_{j=0}^{k^+} \left(L^+_j \x (\R_+ + i \cdot j)\right)
$$
outside $\pi\inv_\C(K)$ for some compact set $K \subset \C$ (see \Cref{fig:lagcob} for the case $k^\pm =0$).

One can use Lagrangian cobordisms to generate an equivalence relation between Lagrangian submanifolds.
This equivalence relation can be packed into a linear-algebraic invariant called the \emph{Lagrangian cobordism group}.
We present here a slightly more detailed definition than in \Cref{def:cobsimple}, which takes into account possible restriction of generators and relations to classes $\mcal L$, $\mcal L_{cob}$:
\begin{defn}\label{defn:lagcobgroup}
    Let $\mcal L$ be a collection of Lagrangians in $X$ and let $\mcal L_{cob}$ be a collection of Lagrangian cobordisms in $X \x \C$ all of whose ends are in $\mcal L$.
    The \emph{Lagrangian cobordism group} of $X$ is the group
    $$
        \Cob(X) \equiv \Cob(X;\mcal L,\mcal L_{cob})  := \Z \vspan{\mcal L}/\sim
    $$
    where the equivalence relation is generated by expressions of the form 
    $$
    \sum L_i^- \sim \sum L_i^+ 
    $$
    if there exists a Lagrangian cobordism in $\mcal L_{cob}$ between the tuples $(L_i^-)$ and $(L_i^+)$.
\end{defn}

It is often the case that one is interested in Lagrangians equipped with some extra data.
In this case, the cobordisms are also required to carry such extra data, and the restriction of data should be compatible.
For instance, in order for Lagrangians to have well-defined Floer cohomology they must satisfy some technical conditions.
These can be quite strong and topological (such as exactness, monotonicity or $\pi_2(X,L)=0$) or weaker and more algebraic (such as being unobstructed).

We will work with a setting rigid enough so that Lagrangian cobordism does not reduce to algebraic topology but flexible enough so that geometric composition of Lagrangian correspondences (see \Cref{sec:lagrangiancorrespondences}) preserve the data.\footnote{Whether a Lagrangian immersion exists or not is governed by a h-principle, and thus reduces to a question in algebraic topology. One can furthermore perturb Lagrangians by a Hamiltonian isotopy to put them in general position and then surger the self-intersections, all while preserving the cobordism class. Therefore, embedded Lagrangian cobordisms with no extra decorations reduces to algebraic topology.}
Namely:
\begin{assumption}\label{assumption:graded_spin}
    All our Lagrangians and cobordisms are \emph{oriented} and \emph{Spin}.\footnote{To obtain orientations on moduli-spaces---and thus a Fukaya category defined over field of characteristic different from $2$---it is enough to assume the Lagrangians are Pin (not necessarily Spin). However, for Calabi-Yau manifolds as in this paper, graded Lagrangians (with respect to the quadratic volume forms we consider) are automatically oriented. Hence, Pin and Spin are equivalent.}
\end{assumption}
\begin{assumption}\label{assumption:generators}
   Lagrangians $L \in \mcal L$ are \emph{embedded} and \emph{weakly-exact} (i.e. $\omega(\pi_2(X,L))=0$).
\end{assumption}
\begin{assumption}\label{assumption:cobordisms}
    Cobordisms $V \in \mcal L_{cob}$ are \emph{immersed} (with clean self-intersections).
\end{assumption}

\begin{remark}
    In this paper we are mainly interested in Lagrangian torus fibers and Lagrangian sections of a smooth Lagrangian torus fibration $X(B) \to B$ over a tropical affine torus $B$.
    In this case, fibers and sections are not just weakly-exact but in fact \emph{topologically unobstructed}: the relative homotopy groups $\pi_2(X, L)$ vanish.
\end{remark}

Putting the above together, we define the following Lagrangian cobordism group:
\begin{defn}\label{defn:lagcobgroup_thispaper}
    Let $B$ be a tropical affine torus and $X(B) \to B$ the corresponding symplectic torus.
    We consider the following two families of Lagrangians:
    \begin{itemize}
        \item $\mcal L$ is the collection of \emph{embedded, weakly-exact, oriented Lagrangian branes} in $X(B)$;
        \item $\mcal L_{cob}$ is the collection of \emph{immersed oriented Lagrangian cobordism branes} in $X(B) \x \R^2$ all of whose ends are in $\mcal L$.
    \end{itemize}
    We define the \emph{Lagrangian cobordism group} of $X(B)$ to be the group
    $$
        \Cob(X(B)) \equiv \Cob(X(B);\mcal L,\mcal L_{cob}) 
    $$
    following the notation of \Cref{defn:lagcobgroup}.
\end{defn}

\begin{remark}
    In Floer theory one often requires Lagrangians to be equipped with a \emph{grading}.
    This ensures the Floer cohomology groups are $\Z$-graded vector spaces, and in particular that the Fukaya category is $\Z$-graded.
    To define gradings one first needs to choose an almost complex structure and a non-zero quadratic volume form 
    $$
    \eta^2 \in H^0((\wedge^n T^*X)^{\otimes 2})
    $$
    which exists as long as $2c_1(X) = 0$.
    This is the case for $X = X(B)$ a torus---in fact, here we have the stronger condition $c_1(X) = 0$.
    By choosing $\eta^2$ to be actually of the form $\eta \otimes \eta$ for $\eta\in H^0(\wedge^n T^*X)$, one ensures that \emph{graded Lagrangians are oriented (and vice versa)}.
    Therefore one could replace the term `oriented' by `graded' in \Cref{defn:lagcobgroup_thispaper}.
    For more details on gradings see \cite{seidel2008fukaya}.
\end{remark}

There are two main constructions of Lagrangian cobordisms: Hamiltonian suspensions and Lagrangian surgeries.
The former is due to Audin-Lalonde-Polterovich \cite{audin1994symplectic} and produces a Lagrangian cobordism between $L_0$ and $L_1$ whenever there is a Hamiltonian isotopy between them.
The second construction is due to Biran-Cornea \cite{biran2013lagrangian} and produces a Lagrangian cobordism between $L_0, L_1$ and $L_2$ whenever $L_0 = L_1\# L_2$ is the result of a Lagrangian surgery---as defined by Lalonde-Sikorav \cite{lalonde1991sous} and Polterovich \cite{polterovich1991surgery}---between transverse Lagrangians $L_1\pitchfork L_2$.

More general surgeries and cobordisms were constructed by Hicks in \cite{hicks2019tropical, hicks2020tropical}.
Hicks constructs a Lagrangian surgery and a corresponding cobordism for any two Lagrangians intersecting along a set which can be locally modelled as the intersection of the zero-section of the cotangent bundle and the graph of the differential of a convex function. 

Let us make his result more precise for further reference.
Let $B$ be a tropical affine manifold.
We denote by $C^{trop} \subset C^0$ the subsheaf of \emph{tropical polynomials}: these are convex\footnote{We remark that our convention for tropical geometry differs from that in \cite{hicks2020tropical}. Namely, for Hicks a tropical polynomial is defined as the \emph{minimum} of a collection of affine functions, whereas we define it as the \emph{maximum}. This translates into his tropical polynomials being concave and ours being convex. Note however that when doing surgeries Hicks considers the graph of the \emph{inverse} $-f$ of a tropical polynomial, thus turning it into a convex function (cf. \cite[Definition 3.15]{hicks2020tropical}).} piecewise-linear functions  with integer slope, i.e. continuous functions $\phi: B \to \R$ which can be locally written as
\begin{equation}\label{eq:tropicalpolynomial}
    \phi = \max \{f_1,\dots,f_k\},\quad f_i \in \Aff(U), \quad  U \subset B.
\end{equation}
Similarly, we will say a continuous function $\Phi: B \to \R$ is a \emph{tropical rational function} if it can be locally written as a difference of two tropical polynomials, i.e.
$$
\Phi = \phi^+ - \phi^- = \max \{f^+_1,\dots,f^+_{k^+}\} - \max \{f^-_1,\dots,f^-_{k^-}\},\quad f^\pm_i \in \Aff(U), \quad  U \subset B.
$$
Given a tropical polynomial $\phi \in C^{trop}(B)$, there exists a smoothing which is convex and which agrees with $\phi$ outside a small neighborhood of its non-smooth locus \cite[Section 3.2]{hicks2020tropical}.
We abuse notation and denote this smoothing also by $\phi$.
Thus, $d\phi$ is well-defined and defines a Lagrangian section in $T^*B$.
The same applies to tropical rational functions.

More generally, global sections $\phi\in H^0(C^{trop}/\Aff)$ define Lagrangian sections $\Gamma(d\phi) \subset T^*B /T^*_\Z B$.
Associated to $\phi$ there is also  a subset
$$
V(\phi) \subset B
$$ 
consisting of those points at which $\phi$ is not smooth.
It is a \emph{tropical hypersurface}: a tropical subvariety of codimension 1, as defined in \Cref{sec:tropicalgeometry}.
Let $\Gamma_0 \subset T^*B/T_\Z^* B$ be the zero-section and
$$
 A := (\Gamma_0 \cup \Gamma(d\phi)) \setminus (\Gamma_0 \cap \Gamma(d\phi))
$$
be the symmetric difference of $\Gamma_0$ and $\Gamma(d\phi)$.
Since the construction of the smoothing is such that it agrees with the original piecewise-linear function outside a small neighborhood of $V(\phi)$, by definition of $A$ we have that $\pi(A)$ lives in a neighborhood of $V(\phi)$.
Hicks proves:
\begin{prop}[\cite{hicks2020tropical}]\label{prop:hicks}
    Given $\phi \in H^0(C^{trop}/\Aff)$, there  exists a Lagrangian submanifold $L_\phi$ such that:
    \begin{itemize}
        \item $L_\phi \subset B_\epsilon(A)$, so in particular $\pi(L_\phi)$ is contained in a small neighborhood of $V(\phi)$;
        
        \item $L_\phi = \Gamma_0 \cup \Gamma(d\phi)$ outside a small neighborhood of $\Gamma_0 \cap \Gamma(d\phi)$;
        
        \item there is a Lagrangian cobordism between $(\Gamma_0,\Gamma(d\phi))$ and $L_\phi$.
    \end{itemize}
\end{prop}

\begin{figure}
    \centering
    \begin{tikzpicture}[thick]

        \begin{scope}[]
\begin{scope}[shift={(3,-1.5)}]
        \draw  (-3.5,9) rectangle (0.5,5);
        \node[scale = 1.5] (v2) at (-1.5,9) {$>>$};
        \node[scale = 1.5] at (-1.5,5) {$>>$};
        \node[rotate = 90,scale = 1.5] (v1) at (-3.5,7) {$>$};
        \node[rotate = 90,scale = 1.5] at (0.5,7) {$>$};
        \end{scope}
        
\draw[red] (3.5,5.5) node (v4) {} .. controls (1.5,5.5) and (1.5,5.5) .. (1.5,3.5);
        \draw[red] (-0.5,5.5) node (v3) {} .. controls (1.5,5.5) and (1.5,5.5) .. (1.5,7.5);
        
\draw[blue] (v3) -- (v4);
        \end{scope}
        
        \begin{scope}[shift={(6.5,0)}]
\begin{scope}[shift={(3,-1.5)}]
        \draw  (-3.5,9) rectangle (0.5,5);
        \node[scale = 1.5] (v2) at (-1.5,9) {$>>$};
        \node[scale = 1.5] at (-1.5,5) {$>>$};
        \node[rotate = 90,scale = 1.5] (v1) at (-3.5,7) {$>$};
        \node[rotate = 90,scale = 1.5] at (0.5,7) {$>$};
        \end{scope}

        \end{scope} 
        
\draw[orange] (8,7.5) .. controls (8,5.5) and (7.5,6) .. (7.5,5.5) node (v5) {};
        \begin{scope}[shift={(17,11)}]
        \draw[rotate=180, orange] (9,7.5) .. controls (9,5.5) and (8.5,6) .. (8.5,5.5) node (v6) {};
        \end{scope}
        \draw[orange] (7.5,5.5) edge  (8.5,5.5) ;

        \node[red] at (0.5,6.5) {$\Gamma(d\phi)$};
        \node[blue] at (2.5,6) {$\Gamma_0$};
        
        \node[orange] at (7,6.5) {$L_\phi$};
        \end{tikzpicture}
    \caption{Hick's surgery inside $T^2 = X(S^1)$. \emph{Left}: the tropical section $\Gamma(d\phi)$ (in red) intersects the zero-section $\Gamma_0$ (in blue) everywhere except where the smoothing has been applied. \emph{Right}: the surgery $L_\phi = \Gamma_0 \# \Gamma(d\phi)$ (in orange) agrees with $\Gamma_0$ and $\Gamma(d\phi)$ outside a neighborhood of their intersection.}
    \label{fig:jeff_surgery}
\end{figure}
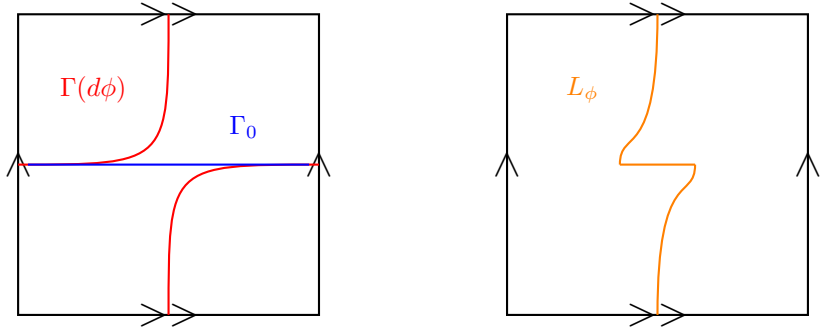

See \Cref{fig:jeff_surgery} for a $2$-dimensional example.

\begin{definition}
    We call the Lagrangian $L_\phi$ the \emph{lift} of the tropical hypersurface $V(\phi)$.
\end{definition}

\begin{remark}\label{rmk:orientation}
    When the base $B$ is orientable (as is the case in this paper, where $B = T^n$ is a tropical affine torus), an orientation of $B$ gives a canonical orientation for Lagrangian sections of $X(B) \to B$.
    If we equip Lagrangian sections with their canonical orientation, then the Lagrangian cobordism of \Cref{prop:hicks} can be upgraded to an \emph{oriented} Lagrangian cobordism such that its two section ends are identified with  $\Gamma(d\phi)$ and $(-1)^n\Gamma_0$ (where, for an oriented Lagrangian $L$, $-{L}$ denotes the same Lagrangian with the opposite orientation).
    The induced orientation on $L_\phi$ is then referred to as its canonical orientation.
\end{remark}
\begin{remark}
    A similar statement to \Cref{prop:hicks} holds for \emph{inverses} $-\phi$ of tropical polynomials. 
    These are piecewise-linear functions with integer slope which are concave instead of convex; equivalently, they can be locally written as the minimum---instead of the maximum, cf. \Cref{eq:tropicalpolynomial}---of a collection of affine functions with integer slope.
    If $-\phi$ is such function then one can define a surgery
    $$
    L_{-\phi} := \Gamma(-d\phi)\#\Gamma_0
    $$
    just as in \cite{hicks2020tropical}.
    The Lagrangian $L_{-\phi}$ satisfies the first two statements of \Cref{prop:hicks}, and a minor adaptation of Hicks' argument shows there is a Lagrangian cobordism between $(\Gamma(-d\phi),\Gamma_0)$ and $L_{\phi}$ (note that now the ordering is $(\Gamma(-d\phi),\Gamma_0)$ instead of $(\Gamma_0,\Gamma(-d\phi))$, but this does not affect the relations in the cobordism group).
    If one includes orientations as in \Cref{rmk:orientation}, then the cobordism can be oriented so that its two section ends are identified with $\Gamma(-d\phi)$ and $(-1)^n {\Gamma}_0$.
\end{remark}
 \subsection{Lagrangian correspondences}\label{sec:lagrangiancorrespondences}
In this section we give some background on Lagrangian correspondences and their relation to Lagrangian cobordism groups.

\begin{definition}
    A \emph{Lagrangian correspondence} between symplectic manifolds $(X_0,\omega_0)$ and $(X_1,\omega_1)$ is a Lagrangian submanifold 
    $$
    L_{01} \subset (X_0 \x X_1, -\omega_0 \oplus \omega_1).
    $$
    (Note the change of sign in $\omega_0$.)
    We will write $L_{01}: X_0 \Rightarrow X_1$ to denote such correspondence.
\end{definition}

Given Lagrangian correspondences $L_{01}: X_0 \Rightarrow X_1$ and $L_{12}: X_1 \Rightarrow X_2$, ideally one would like to be able to `compose' them to obtain a third correspondence $L_{12} \circ L_{01}: X_0 \to X_2$. 
A natural candidate for such composition would be 
$$
    L_{12} \circ L_{01} := \pi_{02}(L_{01} \x_{X_1} L_{12})
$$
but this does not work for two reasons.
First, the fiber product $L_{01} \x_{X_1} L_{12}$ need not be transverse, so that $L_{01} \x_{X_1} L_{12}$ is not even a manifold; second, even if it were transverse, the projection $\pi_{02}$ need not embed $L_{01}\x_{X_1} L_{12}$ into $X_0 \x X_2$ (although it does immerse it, see \cite[Lemma 2.0.5]{wehrheim2010quilted}).

The second issue is more fundamental and cannot be fixed a priori.
The first issue can be circumvented by perturbing either $L_{01}$ or $L_{12}$ so that the fiber product is indeed transverse (as transversality is a generic condition).
That is, if we are allowed to (Hamiltonianly) perturb either of the Lagrangians and we are willing to deal with immersed Lagrangians, then we have a well-defined geometric composition of Lagrangian correspondences.

Now recall that Hamiltonian isotopic Lagrangians are Lagrangian cobordant. 
It follows that, if $L_{i,i+1}: X_i \Rightarrow X_{i+1}, i=0,1$ are two Lagrangian correspondences and $\varphi: X_0 \x X_1 \to X_0 \x X_1$ is a Hamiltonian isotopy such that $L_{01} \x_{X_1} L_{12}$ is transverse, then the Lagrangian cobordism class of  $L_{12} \circ \varphi(L_{01})$ is independent of the choice of $\varphi$.
Moreover, Hanlon-Hicks show in \cite[Proposition C.2]{hanlon2022aspects} that changing the cobordism class of $L_{01}$ (not necessarily by a Hamiltonian isotopy) does not change the cobordism class of $L_{12} \circ L_{01}$ either.
That is, if $L_{01}$ is cobordant to $L_{01}^1,\dots, L_{01}^k$ (as Lagrangians in $X_0 \x X_1$), then $L_{12} \circ L_{01}$ is cobordant to $L_{12} \circ L_{01}^1, \dots, L_{12} \circ L_{01}^k$ (without loss of generality all intersections are transverse---otherwise perturb as before).

The upshot is the following:
\begin{prop}\label{prop:hickshanlon}
    Let $L_{01}: X_0 \Rightarrow X_1$ be Lagrangian correspondence. 
    There is a well-defined group homomorphism
    \begin{align*}
        \Psi_{L_{01}}:\Cob(X_0) &\to \Cob(X_1)\\
        L &\mapsto L_{01} \circ (\varphi_L(L))
    \end{align*}
    where $\varphi_L: X_0 \to X_0$ is \emph{any} Hamiltonian isotopy such that $L \x_{X_0} L_{01}$ is transverse.
\end{prop}
It is important to note that the above is only true in a cobordism group of \emph{immersed Lagrangians and immersed cobordisms}: we can always ensure transversality of fiber products, but the projection need not be an embedding.
However, work of Fukaya shows that, even if immersed, geometric composition of two unobstructed Lagrangian correspondences is again unobstructed \cite{fukaya2017unobstructed}.
That is:
\begin{prop}
    Any unobstructed Lagrangian correspondence $L_{01}: X_0 \Rightarrow X_1$ induces by geometric composition a  group homomorphism $\Cob(X_0) \to \Cob(X_1)$ between cobordism groups of immersed unobstructed Lagrangians and cobordisms.
\end{prop} 
\section{Filtering the Lagrangian cobordism group}
Recall from \Cref{def:pontryagin_intro} that we denote by
\begin{align*}
    \star: \Cobfib(X) \otimes \Cobfib(X) &\to \Cobfib(X)\\
    F_p \star F_q & = F_{m(p,q)}
\end{align*}
the Pontryagin product operation, where $m:B \x B \to B$ denotes the group operation on the base and $X = X(B)$.
Given any subgroup $H \subset \Cobfib(X)$, we denote by $H^{\star k} \subset \Cobfib(X)$ the subgroup of elements of the form 
$$
h_1\star\dots\star h_k,\quad h_i \in H.
$$
We use the convention that $H^{\star 0} = \Cobfib(X)$.

\begin{definition}\label{defn:filtration}
    Let $\Cobfib(X)_{\hom} \subset \Cobfib(X)$ be the subgroup generated by elements of the form $F_p - F_q$.
    We define the \emph{parallelotope filtration} on $\Cobfib(X)$
    as the decreasing filtration 
    $$
        \dots\subset F^1\Cobfib(X) \subset F^0\Cobfib(X) = \Cobfib(X)
    $$
    where $F^i\Cobfib(X) := \Cobfib(X)_{\hom}^{\star i}$.
\end{definition}

In the following lemmas we characterize the first graded pieces of this filtration (c.f. the discussion after \Cref{mainthm:filtration_terminates}):
\begin{lem}\label{lem:F0F1}
    There is an isomorphism $F^0 \Cobfib(X)/F^1 \Cobfib(X) \cong H_0(B)$  
\end{lem}
\begin{proof}
    By definition, we have that $F^0 \Cobfib(X) = \Cobfib(X)$. 
    Now consider the cycle-class map
    \begin{align*}
        \Cobfib(X) & \to H_0(B)\\
        \sum n_i F_{b_i} &\mapsto \sum n_i.
    \end{align*}
    It is clearly surjective, and its kernel---those tuples $\sum n_i F_{b_i}$ such that $\sum n_i=0$ is generated by linear combinations of the form $F_{b} - F_{b'}$. 
    The latter is precisely $F^1 \Cobfib(X)$, thus $H_0(B)\cong \Cobfib(X)/F^1 \Cobfib(X) = F^0 \Cobfib(X) /F^1 \Cobfib(X)$.
\end{proof}
    
\begin{lem}\label{lem:F1F2}
    There is an isomorphism $F^1 \Cobfib(X)/F^2 \Cobfib(X) \cong \Alb(B)$.
\end{lem}
\begin{proof}
    As argued above, the group $F^1\Cobfib(X)$ is nothing but the kernel $\Cobfib(X)_{\hom}$ of the cycle-class map $\Cobfib(X)\to H_0(B)$.
    Sheridan-Smith showed in \cite[Lemma 2.10]{sheridan2021lagrangian} that there exists a (surjective) map $\Cobfib(X)_{\hom} \to \Alb(B)$; essentially, this map takes a tuple $F_b - F_{b'}$ to the difference $b-b'$ in the base $B$ (which is a torus, and thus has an additive structure).
    It follows that the kernel of this map is generated by sums $(F_b - F_{b'}) - (F_c - F_{c'})$ where $b - b' = c - c'$ in the base.
    Similarly, note that $F^2\Cobfib(X)$ is generated by classes of the form
    $$
        (F_s - F_{s'})\star(F_t - F_{t'}) = (F_{m(s,t)} - F_{m(s',t)}) - (F_{m(s,t')} - F_{m(s',t')})
    $$
    and indeed one has $m(s, t)-m(s',t)=s-s'=m(s,t')-m(s',t')$.
    Thus $F^2\Cobfib(X)$ is naturally isomorphic to the kernel of map $\Cobfib(X)_{\hom} \to \Alb(B)$, and thus $\Alb(B)\cong\Cobfib(X)_{\hom}/F^2\Cobfib(X) = F^1\Cobfib(X)/F^2\Cobfib(X)$.
\end{proof}

 \subsection{Constant sections and tropical Lagrangians}\label{sec:constantsections}
Recall from \Cref{sec:proof_sketch} that the strategy to prove \Cref{mainthm:filtration_terminates} is to (i) think of fibers of $X(B) \to B$ as (constant) sections of the fibration $X(B) \to F$, where $F = T_0^* B / T_{\Z,0}^* B$ is a fiber; and (ii) study (fiberwise addition of) these sections using the tropical geometry of the base by turning them into tropical Lagrangians.
In this section we will show: (i) when $B$ is polarized, then so is $F$; (ii) constant sections over a polarized tropical affine torus have associated tropical divisors.

We first explain how to view $F$ as a tropical affine torus.
Let $B = (\R^n/\Lambda, \Z^n)$ be a tropical affine torus (see \Cref{ex:tropicalaffinetorus}).
Then the fiber $F = T^*_0 B / T^*_{\Z,0} B$ can be written as 
$$
F = (\R^n)^\vee / (\Z^n)^\vee
$$
where $(\Z^n)^\vee \subset (\R^n)^\vee$ is the dual lattice to $\Z^n \subset \R^n$ (i.e. those $\alpha \in (\R^n)^\vee$ such that $\alpha(v) \in \Z$ for all $v \in \Z^n$).
We equip $F$ with the tropical affine structure given by the lattice
$$
T_\Z F := \Lambda^\vee = \{\alpha \in (\R^n)^\vee \mid \alpha(\Lambda) \subset \Z\}.
$$
(This is equivalent to declaring that, if $(q_i)$ are local affine coordinates on $B$, then the dual coordinates $(p_i)$ are affine coordinates on the fiber.)
A more detailed explanation of this construction can be found in \Cref{sec:symplecticfourier}.

Let us now recall the definition of a polarization, which goes back to Mikhalkin-Zharkov:
\begin{definition}[\cite{mikhalkin2008tropical}]\label{def:polarized_torus}
    A \emph{polarization} on a tropical affine torus is the data of a map 
    $$
    c: H_1(B;\Z) \to H^0(T^*_\Z B)
    $$
    such that $\vspan{\gamma_1,\gamma_2}:= \int_{\gamma_1} c(\gamma_2)$ defines a symmetric, positive-definite pairing on $H_1(B;\Z)$.
\end{definition}

\begin{example}\label{ex:polarization_on_torus}
    Consider the family of tropical affine $2$-tori of \Cref{fig:polarised_torus} (see also \Cref{ex:polarization_2_torus}).
    Let 
    \begin{align*}
        \gamma_1 &= (\alpha_1+\alpha_2, \alpha_2)\\
        \gamma_2 &= (\alpha_2, \alpha_2 + \alpha_3)
    \end{align*}
    be generators of $H_1(B;\Z)$.
    Then the homomorphism
    \begin{align*}
        c: H_1(B;\Z) &\to H^0(T^*_\Z B)\\
        c(\gamma_i) &= dx_i,
    \end{align*}
    where $(x_1,x_2)$ are the standard coordinates on $B$, is a polarization.
    Indeed, the pairing is clearly positive definite, and furthermore
    $$
        \int_{\gamma_1} c(\gamma_2) = \int_{\gamma_1} dx_2 = \alpha_2 = \int_{\gamma_2} dx_1 = \int_{\gamma_1} c(\gamma_2).
    $$
\end{example}

\begin{lem}
    A tropical affine torus $B$ is polarized if and only if the fiber $F$ is polarized.
\end{lem}
\begin{proof}
    First note that the fiber of $X(B) \to F$ is $B$, so it is sufficient to show that $B$ polarized implies $F$ polarized.
    Now note that there are canonical identifications
    $$
    H_1(B;\Z) \cong \Lambda, \quad  
    H^0(T^*_\Z B) \cong (\Z^n)^\vee
    $$
    for $B = (\R^n/\Lambda,\Z^n)$.
    We view the polarization  as a map $c:\Lambda \to (\Z^n)^\vee$.
    Dualizing this map yields $c^\vee: \Z^n \to \Lambda^\vee$ which is precisely the data of a polarization on $F = ((\R^n)^\vee/(\Z^n)^\vee,\Lambda^\vee)$ (note that if $c$ is symmetric and positive-definite then so is $c^\vee$).
\end{proof}

In what follows we denote by $B$ a polarized tropical affine torus. 
We remind the reader that the results developed in this section (that is, the study of flat sections over a polarized tropical affine torus) will be ultimately applied to the polarized tropical affine torus $F$ (the fiber of $X(B) \to B$).

Let $B$ be a polarized tropical affine $n$-torus and let $c: H_1(B;\Z) \to H^0(T^*_\Z B)$ be its polarization.
Sheridan-Smith proved in \cite[Lemma 3.9]{sheridan2021lagrangian} that such data is equivalent to a Riemannian metric $g$ on $B$ such that $g(\gamma,v) \in \Z$ for all $\gamma \in H_1(B;\Z)$ and all $v \in T_\Z B$.\footnote{We implicitly identify $H_1(B;\Z)$ with its image in $H^0(TB) \cong \Hom(H^0(T^*B),\R)$ via the integration pairing.}
Namely, from a polarization one defines 
\begin{equation}\label{eq:metricfrompolarisation}
    g(\gamma_1,\gamma_2) = \int_{\gamma_1} c(\gamma_2), \quad \gamma_i \in H_1(B;\Z),
\end{equation}
and then extends linearly to $TB \cong H_1(B;\Z)\otimes\R$.

Following \cite[Section 3.7]{sheridan2021lagrangian}, we define the following collection of piecewise-linear functions on $B$.
Let $k \in \N$ and choose a function $\delta: T^*_\Z B \to \R$ such that 
\begin{equation}\label{eq:periodicitydelta}
    \delta(\alpha + kc(\gamma)) = \delta(\alpha), \quad \gamma \in H_1(B;\Z).
\end{equation}
We define a piecewise-linear function $f_{k,\delta} : H^0(TB) \to \R$ by\footnote{Here and in what follows, we view both $TB$ and $T^*B$ as local systems, i.e. the vector bundles $TB$ (resp. $T^*B$) equipped with the flat connection given by $T_\Z B$ (resp. $T^*_\Z B$). Their sections $H^0(TB)$ (resp. $H^0(T^*B)$) are thus  (locally) constant vector fields (resp. $1$-forms).}
\begin{equation}\label{eq:fkdeltadefn}
    f_{k,\delta}(v) = \max_{\alpha\in T^*_\Z B}
    \left\lbrace
        \alpha(v) - \frac{|\alpha^\#|^2}{2k} + \delta(\alpha)
    \right\rbrace
\end{equation}
where $g(\alpha^\#,-)=\alpha$ and $|\alpha^\#|^2 = g(\alpha^\#,\alpha^\#)$ is (the square of) the norm defined by $g$.
This is a piecewise-linear function with integer slope, and furthermore one can check that for any $\gamma \in H_1(B;\Z)$ 
\begin{align}\label{eq:periodicityfdelta}
    f_{k,\delta}(v + \gamma) - f_{k,\delta}(v) &= k\cdot c(\gamma)(v) + \frac{k |\gamma|^2}{2}
\end{align}
is an affine function \cite[Lemma 3.23]{sheridan2021lagrangian}.
It follows that $f_{k,\delta}$ descends to a section of $C^0/\AffZ$.

\begin{lem}\label{lem:appproximationfdelta}
    The sections
    $$
    \frac{k}{2}|-|^2, f_{k,\delta} \in H^0(C^\infty/\Aff)
    $$
    define the same class in $H^1(\Aff)$.
\end{lem}
\begin{proof}
    We have that
    \begin{align*}
        \frac{k}{2}|v+\gamma|^2 - \frac{k}{2}|\gamma|^2 
        &= \frac{k}{2}\left(|v|^2 + 2g(v,\gamma) + |\gamma|^2\right) - \frac{k}{2}|\gamma|^2 \\
        &= \frac{k}{2}|v|^2 + k\cdot c(\gamma)(v) 
    \end{align*}
    where in the second equality we have used the relation between the metric and the polarization, see \Cref{eq:metricfrompolarisation}.
    That is, the function $\frac{k}{2}|-|^2$ has the same quasi-periodicity as $f_{k,\delta}$ (cf. \Cref{eq:periodicityfdelta}).
    Hence, the difference is a well-defined smooth function on $B$, and the result follows from the isomorphism in \Cref{eq:H1Affasquotient}.
\end{proof}

Recall from \Cref{sec:tropicallagrangiansections} that functions on $B$ are intimately related to Lagrangian sections of the torus fibration $X(B) \to B$.
Namely, the groups
\begin{equation}\label{eq:sectionsandsectionsuptoham}
    H^0(C^\infty/\Aff),\quad H^0(C^\infty/\Aff)/H^0(C^\infty)
\end{equation}
parametrize, respectively, Lagrangian sections and Lagrangian sections up to Hamiltonian isotopy.
The symplectic upshot of \Cref{lem:appproximationfdelta} is then the following: the square of the (half-)norm function $\frac{1}{2}|-|^2$ (and all its multiples $\frac{k}{2}|-|^2, k\in \N$) define Lagrangian sections which are Hamiltonian isotopic to tropical Lagrangian sections (the graphs of $df_{k,\delta}$).
\begin{remark}
    The functions $f_{k,\delta}$ appearing in \Cref{lem:appproximationfdelta} are sections of $C^0/\Aff$ (instead of $C^\infty/\Aff$) and hence the differences $f_{k,\delta} - \frac{k}{2}\vert-\vert^2$ are  sections of $C^0$ (instead of $C^\infty$). 
    Note however that the functions $f_{k,\delta}$ are not (equivalence classes of) arbitrary continuous functions but piecewise linear ones.
    Such functions can be smoothed out---being modified only in an arbitrarily small neighborhood of their bending locus---to obtain smooth functions (see \cite[Section 3.2]{hicks2020tropical}).
    Whenever we write $df$ (or consider the class of $f$ in any of the groups in \Cref{eq:sectionsandsectionsuptoham}) for such a piecewise linear function, we mean $df^{sm}$ (or the class of $f^{sm}$) for such a smoothing $f^{sm}$.
\end{remark}
We now show that any linear function on $B$ can be written as a linear combination of (shifts of) the norm function.
Hence, by \Cref{lem:appproximationfdelta}, its associated Lagrangian sections is Hamiltonian isotopic to a linear combination of tropical Lagrangian sections.

\begin{lem}\label{lem:linearfunctionasnorm}
    Let $B$ be a polarized tropical affine torus.
    Let $f \in H^0(C^\infty / \Aff)$ be a section whose lift $\tilde f : H^0(TB) \to \R$ is a linear function (that is, the Lagrangian section $\Gamma(df)$ is constant).
    Then for any $k \in \N$ there exist vectors $w_\pm \equiv w_{\pm,k} \in H^0(TB)$ such that 
    $$
        f(v) = \frac{k}{2}\left(|v + w_+|^2 - |v + w_-|^2\right)
    $$
    in $H^0(C^\infty / \Aff)$.
\end{lem}
\begin{proof}
    Let $\tilde f : H^0(TB) \to \R$ be the linear function
    $$
        \tilde f(v) = \alpha(v) + b
    $$
    where $\alpha \in H^0(T^* B)$ and $b \in \R$.
    Choose $w_\pm \in H^0(TB)$ such that 
    \begin{align*}
        k\cdot c(w_+ - w_-) &= \alpha \\
        \frac{k}{2}\left(|w_+|^2 - |w_-|^2\right) &= b
    \end{align*}
    where we extend $c$ to a map $H_1(B;\R) \cong H^0(TB) \to H^0(T^*B)$ by linearity.
    Then we  have that
    \begin{align*}
        \frac{k}{2}\left(|v + w_+|^2 - |v + w_-|^2\right) & =  \frac{k}{2}\left(2g(v,w_+) + |w_+|^2 - 2g(v,w_-) - |w_-|^2\right)\\
        &= k\cdot c(w_+-w_-)(v) + \frac{k}{2}\left(|w_+|^2 - |w_-|^2\right),\\
        &=\alpha(v) + b \\
        &= \tilde f(v).
    \end{align*}
    Here we have used: the definition of the norm $|-|^2 = g(-,-)$ in the first line; the relation $g(v,w) = c(w)(v)$ between the metric $g$ and the polarization $c$ in the second line (cf. \Cref{eq:metricfrompolarisation}).
\end{proof}

In the following we use the notation
$$
f_{k,\delta, w} := f_{k,\delta}(- + w)
$$
to denote the translation of  $f_{k,\delta}$ by $w \in H^0(TB)$.

\begin{cor}\label{cor:linearfunctionastropical}
    Let $B$ be a polarized tropical affine torus and $f \in H^0(C^\infty / \Aff)$ be a linear section.
    There exist vectors $w_\pm \in H^0(TB)$ such that, for any $k \in \N$ and any $\delta_\pm:H^0(TB) \to \R$ satisfying \Cref{eq:periodicityfdelta}, the tropical rational function
    $$
        f_{k,\delta_+,w_+} - f_{k,\delta_-,w_-}
    $$
    has (after smoothing) the same class as $f$ in $H^1(\Aff)$.
\end{cor}
\begin{proof}
    It follows from \Cref{lem:linearfunctionasnorm} that there exists $w_\pm \in H^0(TB)$ such that  $f$ coincides with 
    $$
        v\mapsto \frac{k}{2}\left(|v + w_+|^2 - |v + w_-|^2\right)
    $$
    in $H^0(C^\infty/\Aff)$.
    Hence, the statement of the Corollary is equivalent to showing that, for any choices of $\delta_\pm:H^0(TB) \to \R$, the sections $\frac{k}{2}|-|^2$ and$f_{k,\delta_\pm}$ define the same class in $H^1(\Aff)$.
    This is precisely the content of \Cref{lem:appproximationfdelta}.
\end{proof}

From the identification between $H^1(\Aff)$ and Lagrangian sections up to Hamiltonian isotopy we get: 
\begin{cor}\label{cor:constantsectionistropical}
    Any constant Lagrangian section is Hamiltonian isotopic to the graph of a tropical rational function.
\end{cor}

In the following Lemma we will show a cobordism relation involving immersed ends.
Note that such relation does \emph{not} make sense in our cobordism group, where ends need to be embedded (\Cref{defn:lagcobgroup_thispaper}).
However, when we fiberwise sum $n+1$ of these cobordism, we do get a valid relation in our cobordism group (see \Cref{lem:transversality}).

\begin{lem}\label{cor:difference_of_flat_sections_codimension}
    Let $f_0,f_1$ be two linear functions and $\Gamma(df_0),\Gamma(df_1)$ the associated constant sections.
    Then 
    $$
    \Gamma(df_1) - \Gamma(df_0) 
    $$
    is cobordant to a linear combination of (possibly immersed) tropical Lagrangians lifting tropical subvarieties of codimension $\geq 1$.
\end{lem}
\begin{proof}
    By \Cref{cor:constantsectionistropical} we have that each $\Gamma(df_i)$ is Hamiltonian isotopic to the graph $\Gamma(d\Phi_i)$ of a tropical rational function $\Phi_i$.
    Let 
    $$
        \Phi_i = \phi_i^+ - \phi_i^-
    $$
    where $\phi_i^\pm$ are tropical polynomials (in fact,  $\phi_i^\pm = f_{k,\delta^i_\pm,w^i_\pm}$, c.f. \Cref{cor:linearfunctionastropical}).
    We then have that
    \begin{align}\label{eq:flat_section_to_hypersurfaces}
        \begin{split}
            \Gamma(df_i) &\sim \Gamma(d\Phi_i)\\
            & = \Gamma(d(\phi_i^+-\phi_i^-))\\
            & = \Gamma(d\phi_i^+)\otimes\Gamma(d(-\phi_i^-))\\
            &\sim (L_{\phi_i^+} + \Gamma_0) \otimes (L_{-\phi_i^-}+\Gamma_0)\qquad \text{(\Cref{prop:hicks,prop:hickshanlon})}\\
            &\sim L_{\phi_i^+} \otimes L_{-\phi_i^-} + L_{\phi_i^+} + L_{-\phi_i^-} + \Gamma_0.
        \end{split}
    \end{align}
    We show in \Cref{lem:transversality} that the tropical polynomials $\phi_i^\pm = f_{k,\delta^i_\pm,w^i_\pm}$ can be chosen so that the tropical hypersurfaces $V(\phi_i^\pm)$ intersect transversely.
    As a consequence, denoting by $\mcal L_{\otimes}$ the fiberwise addition correspondence 
    $$
    \mcal L_{\otimes} = \{(q, p_1, q, p_2, q, p_1+p_2) \mid q \in B, p_i \in T^*_q B\} \subset X \x X \x X,
    $$
    the fiber product 
    $$
    (L_{\phi_1^+} \x L_{-\phi_1^-}) \x_{X\x X} \mathcal L_{\otimes}
    $$
    is transverse.
    Thus, the Lagrangian
    $$
        L_{\phi_1^+} \otimes L_{-\phi_1^-} := \mathcal L_{\otimes} \circ (L_{\phi_1^+} \x L_{-\phi_1^-})
    $$
    is an immersed Lagrangian.
    This means the relation in \Cref{eq:flat_section_to_hypersurfaces} is valid in a cobordism group of immersed Lagrangians modulo immersed unobstructed cobordisms (although not in our original cobordism group, where ends are required to be embedded, see \Cref{defn:lagcobgroup_thispaper}).

    We then see that 
    $$
        \Gamma(df_1) - \Gamma(df_2) \sim L_{\phi_1^+} \otimes L_{-\phi_1^-} - L_{\phi_2^+} \otimes L_{-\phi_2^-} + L_{\phi_1^+} + L_{-\phi_1^-} - L_{\phi_2^+} - L_{-\phi_2^-}\\
    $$
    lives over (an arbitrarily small neighborhood of) the union $\cup_{i,\pm}V(\phi_i^\pm)$, which has codimension $1$.
\end{proof}

\begin{remark}
    A little care must be taken when choosing the functions $f_\pm$ in the proof above if one wants $L_{\pm \phi_i^\pm}$ to be embedded Lagrangians.
    Namely, the surgery $\Gamma(df^\pm)\#(-\Gamma_0)$ constructed by Hicks will only be embedded if the hypersurface $V(f^\pm)$ satisfies some regularity condition.
    
    To be more precise, given a tropical polynomial $f = \max_i \{f_i\}$ and a point $p \in V(f)$, there is a lattice polytope $in_p(f) \subset T^*_p B$ defined as the convex hull of the differentials $df_i(p) \in T^*_{\Z,p} B$ such that $f(p) = f_i(p)$. 
    The condition for Hicks' surgery $\Gamma(df^\pm)\#(-\Gamma_0)$ to be embedded is that, for any $p\in V(f)$, the polytope $in_p(f)$ contains no lattice points other than its vertices \cite[Theorem 3.17]{hicks2020tropical}.

    In \Cref{cor:linearfunctionastropical} we exhibited a whole family of piecewise linear approximations of any constant function, one for each triple $(k,\delta_+,\delta_-)$.
    It is a result of Sheridan-Smith that for large enough $k$ and generic $\delta$, the tropical hypersurface $V(f_{k,\delta})$ is indeed \emph{regular} \cite[Lemma 3.24]{sheridan2021lagrangian},\footnote{In the literature there are related (but different) definitions of smoothness/regularity of a tropical hypersurface: stronger ones such as requiring that $in_b(f)$ has minimal volume, and weaker ones such as requiring that it is simply a simplex.} i.e. the polytopes $in_p(f)$ are simplices containing no lattice points other than its vertices. 
    Choosing large $k$ and generic $\delta$, we can make the Lagrangians $L_{\pm \phi_i^\pm}$  is embedded.
\end{remark}

 \subsection{Transversality of tropical hypersurfaces}
Recall from the previous section that, given any constant Lagrangian section $L = \Gamma(df)$, where $f$ is a linear function, there exist vectors $w_\pm \in TB$ such that the difference of tropical polynomials
$$
    f_{k,\delta_+,w_+} - f_{k,\delta_-,w_-}
$$
defines (after smoothing) the same class as $f$ in $H^1(\Aff)$ (\Cref{cor:linearfunctionastropical}).
This result allows some freedom in the choice of the tropical polynomials, namely it holds for any $k \in \N$ and $\delta_\pm$  satisfying \Cref{eq:periodicitydelta}.
We show in \Cref{lem:transversality} that such freedom is enough to achieve transversality of tropical hypersurfaces, where transversality of tropical hypersurfaces is understood as follows:
\begin{defn}
    Two tropical hypersurfaces are said to intersect 
    \emph{transversely} if each pair of cells intersects transversely.
\end{defn}

Now let $k \in \N$ and consider the family of tropical polynomials $f_{k,\delta}$ of \Cref{eq:fkdeltadefn}.
Note that the moduli-space of possible $\delta$'s is $
\R^N$, where $N = \# \coker(k\cdot c)$, being $c$ the polarization (see \Cref{eq:periodicitydelta}).
We equip $\R^{N}$ with the standard tropical affine structure.

\begin{prop}\label{prop:neighborhood_regular}
    Let $\delta \in \R^N$ be such that $V(f_{k,\delta})$ is regular.
    Let $U \subset \R^N$ be a sufficiently small neighborhood of $\delta$.
    Then for any tropical subvariety $V \subset B$, the set
    $$
    \mathcal T_V = \{\delta' \in U \mid V(f_{k,\delta'}) \pitchfork V\text{ intersect transversely}\}
    $$
    is the complement of a finite union of proper affine subspaces.
\end{prop}
\begin{proof}
    The following idea is inspired by the proofs of 
    \cite[Lemma 3.24, Lemma 3.27]{sheridan2021lagrangian}. 
    For simplicity we will use the notation $f_\delta = f_{k,\delta}$.

    Since transversality is an open condition, if $V(f_{k,\delta}) \pitchfork V$ is transverse then  $\mcal T_V = U$.
    Else,  suppose that $V(f_\delta)$ and $V$ intersect non-transversely at some cells $S_1 \subset V(f_\delta), S_2 \subset V$.
    After translation, we can assume that $\delta= 0$.
    
    Let $\pi_{S_1}:TB \to TB/TS_1$ be the natural projection.
    Affine subspaces parallel to $S_1$ are (locally) parametrized by $TB/TS_1$, and the  set
    $$
        \{ w \in TB/TS_1 \mid \pi_{S_1}^{-1}(w) \cap TS_2 \neq \emptyset\}
    $$
    of translations of $S_1$ that intersect $S_2$ is an affine subspace of $TB/TS_1$, its tangent space being $[TS_2]=\pi_{S_1}(TS_2) \subset TB/TS_1$.
    Furthermore, under the assumption that $S_1$ and $S_2$ intersect non-transversely, the above affine subspace is proper (i.e. $[TS_2] \subsetneq TB/TS_1$).

    We consider the  map
    \begin{align*}
        \Delta_{S_1}: U\subset \R^N \to &TB \xrightarrow{\pi_{S_1}} TB/TS_1\\
        \Delta_{S_1}(\delta') &=\left\lbrack \frac{d}{dt} \bigg|_{t=0}S_{1,\delta+t(\delta'-\delta)}\right\rbrack
    \end{align*}
    recording the direction of translation (mod $TS_1$) of the cell $S_1$ under a sufficiently small perturbation  of $f_\delta$.
    It follows from the paragraph above that if $\Delta_{S_1}(\delta') \notin [TS_2]$ then $S_{1\delta'} \cap S_2 = \emptyset$.

    We show in \Cref{lem:perturbations_surject} that if $V(f_\delta)$ is regular then the map $\Delta_{S_1}$ is surjective onto a neighborhood of the origin.
    Thus $\Delta_{S_1}^{-1}([TS_2])$---which is the collection of $\delta'$'s such that the intersection $S_{1\delta'} \cap S_2$ is non-transverse---is a \emph{proper} affine subspace of $\R^N$ (it is proper because of surjectivity and  $[TS_2] \subsetneq TB/TS_1$).
    Taking the union of these subspaces over all pairs $(S_1,S_2)$ of non-transversely intersecting cells $S_1\subset V(f_\delta), S_2 \subset V$ we obtain a finite union of proper affine subspaces of $\R^{N}$, whose complement is precisely $\mathcal T_{V}$.
\end{proof}

\begin{cor}\label{prop:transversality}
    Let $V \subset B$ be a tropical subvariety.
    For $k\in \N$ large enough, the set
    $$
    \mathcal T^{reg}_V = \{\delta \in \R^N \mid V(f_{k,\delta}) \text{ regular}, V(f_{k,\delta}) \pitchfork V\text{ intersect transversely}\}
    $$
    is dense and open.
\end{cor}
\begin{proof}
    First note that being regular and intersecting transversely are both open conditions, hence $\mathcal T^{reg}_V$ is open.

    To show that it is dense, let $\delta \in \R^N$ be arbitrary.
    We first choose a sequence $\{\delta_n\}_{n\in \N}$ converging to $\delta$ such that $V(f_{k,\delta_n})$ is regular for all $n$ (note that the set of $\delta$'s such that $V(f_{k,\delta})$ is regular is dense for large enough $k$ \cite[Lemma 3.24]{sheridan2021lagrangian}).
    Since the $V(f_{k,\delta_n})$ are regular, \Cref{prop:neighborhood_regular} applies and we have a local description of a neighborhood $U_n$ of $\delta_n$ in $\mathcal T^{reg}_V$.
    It is the complement of a finite union of proper affine subspaces---in particular, it is dense. 
    Let $\delta_{nn} \in U_n$ such that $|\delta_{nn} - \delta_n| < 1/n$.
    Then the sequence $\{\delta_{nn}\}_{n\in \N}$ is contained in $\mathcal T^{reg}_V$ and converges  to $\delta$, showing $\mathcal T^{reg}_V$ is dense.
\end{proof}

The following transversality result is the key to prove our main theorem, \Cref{thm:filtration_terminates}.
For $k\in \N, w^i_\pm \in H^0(TB)$ and $\delta^i_\pm$ as in \Cref{eq:periodicitydelta}, we use the notation
$$
    f_{s_i} = f_{k,\delta^i_{s_i},w^i_{s_i}},\quad s_i \in \{\pm1\}.
$$

\begin{prop}\label{lem:transversality}
Let $f_1,\dots,f_{n+1}$ be  linear functions and $w^{1}_\pm,\dots,w^{n+1}_\pm \in H^0(TB)$ the vectors given by \Cref{cor:linearfunctionastropical}. 
There exist $k \in \N$ and $\delta^1_\pm, \dots, \delta^{n+1}_\pm$ such that:
\begin{enumerate}
    \item The hypersurfaces $V(f_{s_i})$ are regular.
    \item For any choices $s_i \in\{\pm1\}$ and any $j \in \{1,\dots,n+1\}$, the intersection
    $$
        \bigcap_{i=1}^{j}V(f_{s_i})
    $$
    is transverse.
\end{enumerate}
In particular, the intersection $\bigcap_{i=1}^{n+1}V(f_{s_i})$ is empty.
\end{prop}
\begin{proof}
    We first choose $\delta^1_\pm$ such that the hypersurfaces $V(f_{s_1})$ are regular (recall this is a generic condition for large enough $k$).
    We then choose $\delta^2_\pm$ such that $V(f_{s_2})$ are regular and the intersections
    $$
    V(f_{s_1}) \pitchfork V(f_{s_2})
    $$
    are transverse.
    Note that the possible choices of such $\delta^2_+$ is the intersection 
    $$
        \mathcal T^{reg}_{V(f_{k,\delta^1_{+},w^1_{+}})} \cap \mathcal T^{reg}_{V(f_{k,\delta^1_{-},w^1_{-}})}
    $$
    and similarly for $\delta^2_-$.
    This is the intersection of two dense open subsets by \Cref{prop:transversality}, hence it is non-empty.

    We can then iterate this process and choose $\delta^k_\pm$ such that the intersections
    $$
        \left(\bigcap_{i=1}^{k-1} V(f_{s_i})\right) \pitchfork V(f_{s_k})
    $$
    are transverse for all $s_1,\dots,s_{k-1} \in \{\pm1\}$.
    By transversality the set-theoretic intersection 
    $$
        V = \bigcap_{i=1}^{k-1} V(f_{s_i})
    $$
    is a tropical subvariety \cite[Section 4.2]{mikhalkin2006tropical}.
    Hence we can again apply \Cref{prop:transversality} to guarantee the set of such $\delta^k_\pm$ is the intersection of ($2^{k-1}$) dense open subsets---in particular,  non-empty.
\end{proof}

We conclude by showing the surjectivity result claim in the proof of \Cref{prop:neighborhood_regular}:
\begin{lem}\label{lem:perturbations_surject}
    Let $V(f)$ be a regular tropical hypersurface and $S \subset V(f)$ any cell.
    Then the map $\Delta_S$ is surjective onto a neighborhood of the origin.
\end{lem}
\begin{proof}
    Let $p \in V(f)^{(n-1-k)}$ be a point in the $(n-1-k)$-stratum of $V(f)$.
    If $\psi: U \subset B \to \R^n$ is an affine chart around $p$, then by regularity of $V(f)$ there exists $A \in \GL(n;\R)$ such that
    $$
    (A \circ \psi)(V(f)) = \mathcal P_{k} \x \R^{n-1-k},\quad (A\circ\psi)(p)=0
    $$
    where $\mathcal P_{k} \subset \R^{k+1}$ is the $k$-dimensional tropical pair of pants.
    (Note that a priori $A$ has real coefficients, hence $A \circ \psi$ is not an affine chart, but we do not need that for the proof.)
    It follows that the map $\Delta_S$ is surjective onto a neighborhood of the origin if and only if the map 
    $$
        \Delta_{\{0\}}: \R^{k+2} \to T_0 \R^{k+1} \cong \R^{k+1}
    $$
    associated to $\mathcal P_{k}$ is surjective. 
    We have that $\mathcal P_{k} = V(h_0)$ for
    $$
        h_\delta = \max\{\delta_0,x_1+\delta_1,\dots,x_{k+1}+\delta_{k+1}\}.
    $$
    A simple computation shows that the map $\Delta_{\{0\}}$ is given by
    $$
        (\delta_0,\dots,\delta_{k+1}) \mapsto (\delta_1-\delta_0,\dots,\delta_{k+1}-\delta_0)
    $$
    which is clearly surjective.
\end{proof}

\subsection{The filtration terminates}
\begin{lem}\label{prop:fiberwise_sum_increases_codimension}
    For any $k \in \{1,\dots,n+1\}$ and flat Lagrangian sections $\Gamma_1^\pm,\dots,\Gamma_k^\pm$, the Lagrangian cobordism class of 
    $$
        \bigotimes_{i=1}^k (\Gamma_i^+ - \Gamma_i^-)
    $$
    is represented by a linear combination of (possibly immersed) tropical Lagrangians lifting subvarieties of codimension at least $k$.
\end{lem}
\begin{proof}
    We proceed by induction.
    The result is true for $k=1$ by \Cref{cor:difference_of_flat_sections_codimension}. 
    Suppose that it holds for some $k$. 
    Using again \Cref{cor:difference_of_flat_sections_codimension}  we have that $\Gamma_{k+1}^+ - \Gamma_{k+1}^-$ is cobordant to a lift of tropical subvarieties of codimension at least $1$.
    These subvarieties can be arranged to be transverse to all the subvarieties appearing in $\otimes_{i=1}^k (\Gamma_i^+ - \Gamma_i^-)$ (cf. \Cref{lem:transversality}).
    It then follows that the fiberwise sum
    $$
    \left(\otimes_{i=1}^k (\Gamma_i^+ - \Gamma_i^-)\right) \otimes (\Gamma_{k+1}^+ - \Gamma_{k+1}^-)
    $$
    lives over tropical subvarieties of codimension at least $k+1$.
\end{proof}

\begin{thm}\label{thm:filtration_terminates}
    Let $B$ be a polarized tropical affine $n$-torus and $X = X(B)$ the corresponding symplectic manifold.
    Then the decreasing filtration 
    $$
        \dots \subset \Cobfib(X)_1 \subset \Cobfib(X)_0 = \Cobfib(X)
    $$
    given by \Cref{defn:filtration} satisfies $\Cobfib(X)_{n+1} = 0$.    
\end{thm}
\begin{proof}
    Let
    $$
    \mathbb{L} = (F_1^+ - F_1^-) \star_{\pi} \dots \star_{\pi} (F_{n+1}^+ - F_{n+1}^-)
    $$
    be a generator of $\Cobfib(X)_{n+1}$.
    Thinking of each $F_i^\pm$ as a flat section $\Gamma_i^\pm$ of $\pi':X(B) \to F$ and recalling that the Pontryagin product $\star_{\pi}$ is fiberwise addition $\otimes_{\pi'}$ of the dual fibration $\pi': X(B) \to F$ (c.f. \Cref{eq:pontryagin_vs_fiberwise}), we have that
    \begin{align*}
        \mathbb{L} &= (F_1^+ - F_1^-) \star_{\pi} \dots \star_{\pi} (F_{n+1}^+ - F_{n+1}^-)\\
        & = (F_1^+ - F_1^-) \otimes_{\pi'} \dots \otimes_{\pi'} (F_{n+1}^+ - F_{n+1}^-).
    \end{align*}
    Now apply \Cref{prop:fiberwise_sum_increases_codimension} with $\Gamma_i^\pm = F_i^\pm$ to show that the above Lagrangian lives over tropical subvarieties of codimension at least $n+1$, hence vanishes.

\end{proof}
 
\section{Fourier transform}
\subsection{Fourier transforms in algebraic geometry}\label{sec:fourier}
The Fourier transform is a powerful tool relating the algebraic geometry of an abelian variety to that of its dual.
It was first introduced by Liebermann at the level of cohomology (see e.g. \cite{kleiman1968algebraic}); Mukai \cite{mukai1981duality} considered it as a transformation at the level of derived categories of coherent sheaves; and later Beauville studied it at the level of Chow groups \cite{beauville1986anneau,beauville2006quelques}.

The \emph{dual abelian variety of $A$} is an abelian variety parametrizing isomorphism classes of degree-zero line bundles on $A$.
It is often denoted by $A^\vee$.
The product $A \x A^\vee$ admits the following universal line bundle:
\begin{defn}\label{defn:poincarebundle}
    The \emph{Poincare bundle} on $A \x A^\vee$ is the line bundle 
    $$
    \mcal L_P = \mcal O(D_P) \to A \x A^\vee
    $$
    uniquely determined (up to isomorphism) by the following properties:
    \begin{enumerate}
        \item for any $w \in A^\vee$, the restriction of $\mcal L_P$ to $A \x \{w\}$ is a degree zero line bundle on $A$;
        \item the restriction of $\mcal L_P$ to $\{0\} \x A^\vee$ is trivial.
    \end{enumerate} 
\end{defn}
\begin{rmk}\label{rmk:Aveevee}
    The Poincare bundle gives an isomorphism $A \cong (A^\vee)^\vee$.
\end{rmk}

One then defines the \emph{Fourier transform} $\mcal F$ as the Fourier-Mukai transform associated to the kernel $\mcal L_P$, i.e.
\begin{align}\label{eq:algebraic_fourier}
    \mcal F := \Phi_{\mcal L_P}: D^bCoh(A) & \to D^bCoh(A^\vee)\\
    \mcal F(F) &:= (\pi_{A^\vee})_*((\pi_A)^*(F) \otimes \mcal L_P)
\end{align}
where $\pi_A,\pi_{A^\vee}$ are the projections of $A \x A^\vee$ to its two factors.
At the level of Chow groups, one can think of $\mcal F$ as 
\begin{align*}
    \mcal F_*: \CH^*(A) & \to \CH^*(A^\vee)\\
    \mcal F_*(Z) &:= (\pi_{A^\vee})_*((\pi_A)^*(Z) \cap D_P)
\end{align*}
where $\mcal L_P = \mcal O(D_P)$ and 
$$
\cap:\CH^*(A\x A^\vee)^{\otimes 2} \to \CH^*(A\x A^\vee)
$$
denotes the intersection product on cycles.

\subsection{Fourier transforms in symplectic geometry}\label{sec:symplecticfourier}
We now construct the mirror to $\mcal F$.
For this one must first understand what is the mirror to $A^\vee$.
Before proceeding formally let us make the following heuristic observation.
According to the SYZ philosophy, the mirror to a symplectic manifold $\pi:X \to B$ together with a given Lagrangian torus fibration (and a section) should be the complex space $Y = TB/T_\Z B$.
A pair 
$$
(F_b,\pi_1(F_b)\to S^1)
$$
of a Lagrangian torus fiber $F_b = \pi\inv(b)$ equipped with an $S^1$-local system can be thought of as an element in the group
$$
\Hom(\pi_1(F_b),S^1) \cong H^1(F_b;S^1) \cong T_{\Z,b} B \otimes S^1 \cong T_b B / T_{\Z,b} B
$$
which is a subset of $Y$.
That is: \emph{Lagrangian torus fibers equipped with local systems correspond to points in the mirror $Y$}.
Similarly, Lagrangian sections of $\pi$ should correspond to line bundles on $Y$.\footnote{To see this, recall there is an identification between $H^1(\Aff)$ and Hamiltonian isotopy classes of Lagrangian sections, see \Cref{eq:H1Affasquotient}.
It is a feature of the construction of the mirror space $\Pi: Y(B) \to B$ (see \cite[Section 2]{abouzaid2014family} that there is a map of sheaves of abelian groups $\Pi^*: \Aff \to \mcal O^*(Y(B))$, which sends an affine function $a + bx$ to a monomial $t^a z^b$. Thus there is a group homomorphism $H^1(B,\Aff) \to H^1(Y(B),\mcal O^*(Y(B))$. Given a Lagrangian section corresponding to a class in $H^1(\Aff)$, the corresponding line bundle on $Y(B)$ is the one classified by the image of this class under this map.} 
A homological mirror symmetry equivalence $\mcal Fuk(X) \simeq D^bCoh(Y)$ is then expected to map Lagrangian torus fibers to skyscraper sheaves of points, and Lagrangian sections to line bundles.

By definition, $A^\vee$ parametrizes isomorphism classes of degree zero line bundles on $A$.
The heuristics above then suggest that, if $\pi':X' \to B'$ is the mirror Lagrangian torus fibration to $A^\vee$, then its Lagrangian torus fibers should correspond to (Hamiltonian isotopy classes of) sections of $X \to B$ which are homologous to the zero-section (these are the sections corresponding to degree-zero line bundles).
Such sections are precisely constant sections.
Similarly, since $(A^\vee)^\vee \cong A$ (see \Cref{rmk:Aveevee}) we have that $A$ parametrizes isomorphism classes of degree-zero line bundles on $A^\vee$.
It follows that constant sections of $X' \to B'$ should correspond to Lagrangian torus fibers of $\pi:X \to B$.
It is then natural to propose
$$
\pi':X \to T_0^*B/T_{\Z,0}^*B
$$
as the homological mirror to $A^\vee$.

Let us make the above heuristic argument  precise.
We start with  a tropical affine torus $B = T^n$.
This corresponds to the data
$$
(V,\Lambda_1,\Lambda_2)
$$ 
of an $n$-dimensional real vector space $V$ together with two full-rank lattices $\Lambda_1, \Lambda_2 \subset V$.
Namely, to such data one associates the tropical affine torus\footnote{This is equivalent to the definition given in \Cref{ex:tropicalaffinetorus}: choosing a basis for  $\Lambda_2$ as a basis for $V$, we have an isomorphism $V \cong \R^n$ such that $\Lambda_2$ is identified with $\Z^n$.}
$$
(B,T_\Z B)  = (V/\Lambda_1, \Lambda_2).
$$
For the following definition we use the notation $V^\vee = \Hom_\R(V,\R)$ for the dual vector space and $\Lambda^\vee = \{\alpha \in V^\vee \mid \langle \alpha,\Lambda \rangle \subset \Z\}$ for the abelian group of covectors in $V^\vee$ that are integer-valued on $\Lambda$.
\begin{definition}\label{def:dualtropicalaffinetorus}
    Given a tropical affine torus $B = (V,\Lambda_1,\Lambda_2)$, the \emph{dual tropical affine torus} is $B^\vee := (V^\vee,\Lambda_2^\vee,\Lambda_1^\vee)$.
\end{definition}

Recall that, given a tropical affine torus $B \equiv (V,\Lambda_1,\Lambda_2)$, the associated symplectic torus $X(B)$ and complex torus $Y(B)_\C$ are the products
\begin{align}
    X(B) &= T^*B / T^*_\Z B = V/\Lambda_1 \x V^\vee/\Lambda_2^\vee,\quad \omega = \sum_i dq_i \wdg dp_i \label{eq:X(B)}\\
    Y(B)_\C &= TB/T_\Z B = V/\Lambda_1 \x V/\Lambda_2,\quad J(v_1,v_2) = (-v_2,v_1).\label{eq:Y(B)}
\end{align}
Here, $(q_i)$ are local coordinates on $B$ and $(p_i)$ are the dual coordinates on the fiber, while $J$ denotes the complex structure on $Y(B)_\C$ (whose tangent space we identify with $TY(B)_\C \simeq V \oplus V$).
These symplectic and complex manifolds are SYZ mirror pairs and, when passing to the rigid-analytic world, homologically mirror to each other \cite{kontsevich2001homological,fukaya2002mirror,abouzaid2021homological}.

\begin{rmk}\label{rmk:dualtorusisfiber}
    The dual tropical affine torus $B^\vee$ should be thought of as the fiber of the fibration $X(B) \to B$.
    Indeed, each fiber $F_b := T^*_b B /T^*_\Z B$ is the quotient of $T^*_bB \cong V^\vee$ by the lattice $T^*_{\Z,b} B \cong \Lambda_2^\vee$, which identifies $F_b \simeq V^\vee/\Lambda_2^\vee$.
    If $(q_i)$ are local affine coordinates on $B$ then one defines the dual coordinates $(p_i)$ as being affine coordinates on the fiber, and the corresponding lattice of integer vectors on the fiber is then canonically identified with $\Lambda_1^\vee$.
\end{rmk}

\begin{lem}\label{lem:SYZ_dual}
    The SYZ mirror pair $(X(B^\vee),Y(B^\vee)_\C)$ satisfies the following properties:
    \begin{enumerate}
        \item $X(B^\vee)$ is naturally symplectomorphic to $X(B)^-$;
        \item $Y(B^\vee)_\C\cong Y(B)^\vee_\C$.
        \item Let $Y(B)$ be the rigid-analytic space living over $B$, as defined by \cite{kontsevich2006affine}. When $B$ is polarized, $Y(B)$ is the analytification of an abelian variety, which we denote still by $Y(B)$. Then $Y(B)^\vee$ can be identified with $Y(B^\vee)$.
    \end{enumerate}
\end{lem}
\begin{proof}
    \begin{enumerate}
        \item Using \Cref{eq:X(B)} and \Cref{def:dualtropicalaffinetorus} we see that
        $$
        X(B^\vee) 
        = 
        V^\vee/\Lambda_2^\vee \x (V^\vee)^\vee/(\Lambda_1^\vee)^\vee 
        \cong
        V^\vee/\Lambda_2^\vee \x V/\Lambda_1.
$$
        Now note that the symplectic form on $X(B^\vee)$ is given by (see \Cref{eq:X(B)})
        $$
            \omega_{X(B^\vee)} = \sum_i dp_i \wdg dq_i = -\omega_{X(B)} = \omega_{X(B)^-}.
        $$
        It follows that the natural map swapping the factors
        $$
            X(B^\vee) \cong V^\vee/\Lambda_2^\vee \x V/\Lambda_1 \xra{\sim} V/\Lambda_1 \x V^\vee/\Lambda_2^\vee = X(B)^-
        $$
        is a symplectomorphism.
        \item From \Cref{eq:Y(B)} and \Cref{def:dualtropicalaffinetorus} we have
        $$
        Y(B^\vee)_\C = V^\vee/\Lambda_2^\vee \x V^\vee/\Lambda_1^\vee \cong (V/\Lambda_1 \x V/\Lambda_2)^\vee = Y(B)^\vee_\C. 
        $$
        \item This is equivalent to \cite[Proposition 4.7]{foster2018non}, where the authors show the statement for Berkovich analytic spaces.
    \end{enumerate}
\end{proof}

\Cref{lem:SYZ_dual} says that the mirror to $Y(B)^\vee$ is the \emph{same} manifold $X(B)$ but with  a \emph{different} Lagrangian torus fibration $X(B) \to V^\vee/\Lambda_2^\vee$ (and consequently the \emph{opposite} symplectic form $\omega_{X(B^\vee)} = -\omega_{X(B)}$).
The homological mirror symmetry functor\footnote{We are being slightly imprecise here. The homological mirror symmetry functor constructed by Abouzaid \cite{abouzaid2014family,abouzaid2017family,abouzaid2021homological} relates the Fukaya category $\mcal Fuk(X(B))$ to the derived category $D^bCoh^{an}(Y)$ of coherent \emph{analytic} sheaves on a rigid-analytic mirror $Y$ defined over the Novikov field. We are instead working with a (projective) algebraic torus $Y(B)$ and the category $D^bCoh(Y(B))$ of coherent algebraic sheaves on $Y(B)$. The passage between analytic and algebraic sheaves poses no problem due to the rigid-analytic GAGA principle: for a rigid-analytic $Y = Y(B)^{an}$ which is the analytification of a projective variety $Y(B)$, one has that $D^bCoh^{an}(Y) \simeq D^bCoh(Y(B))$ \cite{fresnel2012rigid}.}
$$
\Phi_B: \mcal Fuk(X(B)) \to D^bCoh(Y(B))
$$
sends fibers (respectively, Lagrangian sections) of $X(B) \to B$ to skyscraper sheaves of points (respectively, line bundles) on $Y(B)$.
\begin{definition}\label{defn:symplecticfourier}
    The \emph{(symplectic) Fourier transform} $\mathfrak F$ is the equivalence 
$$
    \begin{tikzcd}
        \mcal Fuk(X(B)) \arrow[r,dashed,"\mathfrak F"] \arrow[d,"\sim","\Phi_B"'] & \mcal Fuk(X(B)^\vee) \\
        D^bCoh(Y(B)) \arrow[r,"\sim","\mcal F"'] & D^bCoh(Y(B)^\vee) \arrow[u,"\sim","\Phi_{B^\vee}\inv"']
    \end{tikzcd}
    $$
\end{definition}

We will now obtain a purely symplectic description of $\mathfrak F$.
Recall that its algebraic counterpart  is the equivalence 
$$
\mcal F = \Psi_{\mcal L_P}: D^bCoh(Y) \to D^bCoh(Y^\vee)
$$
induced by the Fourier-Mukai kernel $\mathcal L_P \in D^bCoh(Y \x Y^\vee)$ given by the Poincare bundle (see \Cref{eq:algebraic_fourier}).
The two properties defining $\mathcal L_P$ (see \Cref{defn:poincarebundle}) tell us that, if
$$
    D^bCoh(Y) \xra{\Psi_{\mcal L_P, Y \to Y^\vee}} D^bCoh(Y^\vee), \quad D^bCoh(Y^\vee) \xra{\Psi_{\mcal L_P, Y^\vee \to Y}} D^bCoh(Y)
$$
are the functors obtained from $\mcal L_P$, then
$$
    \Psi_{\mcal L_P, Y^\vee \to Y}(\mcal O_a) = a, \quad \Psi_{\mcal L_P, Y \to Y^\vee}(\mcal O_0) = \mcal O_{Y^\vee}.
$$
Recalling from \cite{mukai1981duality} that 
$$
\Psi_{\mcal L_P, Y^\vee \to Y} \circ \Psi_{\mcal L_P, Y \to Y^\vee}  = (-1_Y)^*[g]
$$
where $-1_Y:Y \to Y $ is the inversion map $(-1_Y)(y) = -y$ and $g$ is the dimension of $Y$, we see that $\Psi_{\mcal L_P, Y \to Y^\vee}$ is such that 
\begin{align}
        \Psi_{\mcal L_P, Y^\vee \to Y}(\mcal O_a) &= a \label{eq:properties_FM_sky}\\
        \Psi_{\mcal L_P, Y^\vee \to Y}(\mcal O_{Y^\vee}) &= \mcal O_0[g] \label{eq:properties_FM_str}.
\end{align}

The following result implies that the action of a functor on skycraper sheaves characterizes it uniquely up to tensoring with a line bundle:
\begin{prop}\label{prop:action_on_sky_characterizes}
    Let $Y_1,Y_2$ be smooth projective varieties and $\Phi: D^bCoh(Y_1) \to D^bCoh(Y_2)$ an equivalence.
    If $\Phi$ sends skyscraper sheaves of points to skyscraper sheaves of points (i.e. $\Phi(\mcal O_y) = \mcal O_{f(y)}$ for some $f: Y_1 \to Y_2$) then:
    \begin{enumerate}
        \item $f$ is an isomorphism;
        \item $\Phi = L \otimes f_*(-)$ for some line bundle $L \in \Pic(Y)$.
    \end{enumerate}
\end{prop}
\begin{proof}
    See \cite[Corollary 5.23]{huybrechts2006fourier}.
\end{proof}

In particular, we have:
\begin{lem}\label{lem:properties_Poincare_characterize}
    The two properties in \Cref{eq:properties_FM_sky,eq:properties_FM_str} characterize $\Psi_{\mcal L_P, Y^\vee \to Y}$ (and hence $\Psi_{\mcal L_P, Y \to Y^\vee}$) uniquely.
\end{lem}
\begin{proof}
    Suppose that $\Psi := \Psi_{\mcal L_P, Y^\vee \to Y}$ and $\tilde\Psi$ both satisfy \Cref{eq:properties_FM_sky,eq:properties_FM_str}. 
    Then 
    $$
    \chi:= \tilde\Psi\inv \circ \Psi: D^bCoh(Y) \to D^bCoh(Y)
    $$
    is an autoequivalence of $D^bCoh(Y)$ such that 
    \begin{align}
        \chi(\mcal O_a) &= \mcal O_a \label{eq:sky_to_sky}\\
        \chi(\mcal O_{Y^\vee}) &= \mcal O_{Y^\vee} \label{eq:str_to_str}.
    \end{align}
    It follows from \Cref{eq:sky_to_sky} and \Cref{prop:action_on_sky_characterizes} that $\chi = L \otimes (-)$ for some line bundle $L \in \Pic(Y)$.
    Using \Cref{eq:str_to_str} we get
    $$
    \mcal O_{Y^\vee} = \chi(\mcal O_{Y^\vee}) = L\otimes  \mcal O_{Y^\vee} = L
    $$
    concluding that $\chi =\id$.
\end{proof}

Let us first forget about gradings.
Then, using homological mirror symmetry \cite{abouzaid2021homological} for $Y=Y(B)$ and $X(B)$ (resp. $Y^\vee$ and $X(B^\vee)$)  together with \Cref{lem:properties_Poincare_characterize}, this means that any equivalence $\mathfrak F:\mcal Fuk(X(B)) \to \mcal Fuk(X(B^\vee))$ such that
\begin{align}
    \mathfrak F(F_0) &= F_0\\
    \mathfrak F(\Gamma_\alpha) &= \Gamma_{-\alpha} \label{eq:sections_shift}
\end{align}
will be unique and will be the mirror to the Fourier-Mukai transform $\mcal F$.
Here we denoted by $F_0$ the fiber of $X(B) \to B$ over $0 \in B$, and by $\Gamma_\alpha$ the flat Lagrangian section corresponding to the constant $1$-form $\alpha\in H^0(T^*B)$---these are the mirror Lagrangians to the skyscraper sheaf of the origin $\mcal O_0$ and degree-zero line bundles $a \in \Pic^0(Y)$, respectively.

Let $\iota: X(B) \to X(B)^\vee$ be the symplectomorphism 
\begin{align}\label{eq:iota}
    \begin{split}
        \iota:X(B) &\to X(B)^\vee\\
        (q_B,p_B) &\mapsto (q_{B^\vee}, p_{B^\vee}) =  (-p_B,q_B).
    \end{split}
\end{align}
It is then clear that $\iota(F_0) = F_0$ and $\iota(\Gamma_\alpha) = \Gamma_{-\alpha}$, so that $\mathfrak F$ is the functor induced by $\iota$ (up to gradings).

To understand the shift by $g$ in \Cref{eq:sections_shift} we must enhance our Lagrangians to \emph{graded} Lagrangians and turn $\iota$ into a \emph{graded} symplectomorphism, cf. \cite{seidel2000graded}.
Namely, note that the quadratic volume forms used to define gradings in $X(B)$ and $X(B^\vee)$ are, respectively \cite{abouzaid2014family}
\begin{align*}
    \Omega^2 &= (\wedge_j (dq_j + i dp_j))^{\otimes 2}\\
    (\Omega^\vee)^2 &= (\wedge_j (dp_j + idq_j))^{\otimes 2}.
\end{align*}
The map $\iota$ of \Cref{eq:iota} does \emph{not} preserve the quadratic volume forms, but rather 
$$
\iota_*\Omega^2 = (\wedge_j (dq_j - i dp_j))^{\otimes 2} = (-i)^{2g} (\Omega^\vee)^2.
$$
Thus, if $W \subset T_pX(B)$ is a graded Lagrangian subspace with grading $t \in \R$ (with respect to $\Omega^2_p$), then the Lagrangian subspace $\iota_*(W) \subset T_{\iota(p)}X(B^\vee)$ can be graded by $t-g/2$ (with respect to $(\Omega^\vee)^2_{\iota(p)}$).
This yields a graded symplectomorphism such that, for a graded Lagrangian  $(L, \tilde\theta_L : L \to \R)$,\footnote{Recall this means that $\exp(2\pi i \tilde\theta_L) = \theta_L$, where $\theta_L(p) = \Omega^2(vol^2_p) \in S^1$ is the squared phase map.} we have 
\begin{align}\label{eq:gradediota}
    \begin{split}
        \varphi_\iota: \mcal Fuk(X(B)) &\to \mcal Fuk(X(B^\vee))\\
        (L, \tilde \theta_L) &\mapsto (\iota(L), \iota_*\tilde\theta_L - g/2),
    \end{split}
\end{align}
where $\iota_*\tilde\theta_L = \tilde\theta_L \circ \iota\inv$.

The family Floer functor $\mcal Fuk(X(B)) \to D^bCoh(Y)$ sends graded Lagrangians to (complexes of) coherent sheaves on $Y(B)$.
By construction, one has that:
\begin{itemize}
    \item The mirror to degree-zero line bundles $a \in \Pic^0(Y)$ are flat Lagrangian sections $\Gamma_\alpha$ equipped with the constant grading $\tilde\theta_{\Gamma_\alpha} \equiv 0$ (note that for these sections the squared phase map $\theta_{\Gamma_\alpha}$ is constant and equal to $1 \in S^1$, hence the above grading makes sense as a constant).
    \item The mirror to skyscraper sheaves of points $\mcal O_y$ are Lagrangian torus fibers $F_b$ (equipped with local systems) and with grading $\tilde\theta_{F_b} \equiv g/2$ (note that fibers also have constant squared phase map $\theta_{F_b} \equiv e^{ig\pi}$).
\end{itemize}
Equipping fibers and sections with this grading, the graded symplectomorphism $\iota$ of \Cref{eq:gradediota} sends
\begin{equation}\label{eq:sympl_fourier_on_fibers}
    (F_0, \tilde\theta_{F_0} \equiv g/2) \mapsto (F_0, 0)
\end{equation}
and $(F_0, 0)$ coincides with the $0$-section of $X(B^\vee)$ with its standard grading $\tilde\theta^\vee_{\Gamma_0} \equiv 0$.
Similarly, we have that
\begin{equation}\label{eq:sympl_fourier_on_sections}
(\Gamma_\alpha, \tilde\theta_{\Gamma_\alpha} \equiv 0) \mapsto (\Gamma_{-\alpha}, -g/2) = (\Gamma_{-\alpha}, g/2)[g]
\end{equation}
and $(\Gamma_{-\alpha}, g/2)$ is nothing but a fiber of $X(B^\vee) \to B^\vee$ with its standard grading.
Note that \Cref{eq:sympl_fourier_on_fibers,eq:sympl_fourier_on_sections} are precisely the mirror (dual) equations to \Cref{eq:properties_FM_sky,eq:properties_FM_str}. 
By \Cref{lem:properties_Poincare_characterize} and homological mirror symmetry \cite{abouzaid2021homological} the action on these objects fully specifies the functor, hence:

\begin{thm}\label{prop:symplecticfourier}
    The symplectic Fourier transform 
    $$
    \mathfrak F: \mcal Fuk(X(B)) \to \mcal Fuk(X(B^\vee))
    $$
    defined in \Cref{defn:symplecticfourier} coincides with the functor $\varphi_\iota$ induced by the graded symplectomorphism $\iota$ of \Cref{eq:gradediota}.
\end{thm}

\section{Ring structures}
\subsection{Pontryagin product}\label{sec:pontryagin}
Let $\mathbf k$ be a field of characteristic zero and 
$$
\Lambda = \left\lbrace \sum_{i\geq0} a_i T^{\lambda_i} \mid a_i \in \mathbf{k},\quad \lim_{i\to\infty} \lambda_i =+\infty\right\rbrace
$$
denote the Novikov field.
Let $\log : \Lambda^\x \to \R$ be the log-map
$$
    \log\left(\sum_{i\geq0} a_i T^{\lambda_i}\right) = \min\lbrace \lambda_i \mid a_i \neq 0\rbrace.
$$
Here $\Lambda^\x \subset \Lambda$ is the multiplicative group of units.
We also write $\log: (\Lambda^\x)^n \to \R^n$ for the component-wise log map.

Let $L \subset (\Lambda^\x)^n$ be a rank-$n$ torsion-free subgroup of $(\Lambda^\x)^n$ such that $\log(L) \subset \R^n$ is a rank-$n$ lattice. 
The quotient 
\begin{equation}\label{eq:Aisquotient}
    A = (\Lambda^\x)^n / L
\end{equation}
is, when polarized, an \emph{abelian variety} of dimension $n$.
The projection
$$
\log : A \to \R^n/\log(L) \cong T^n
$$
gives $A$ the structure of a non-Archimedean torus fibration over $T^n$, c.f. \cite{nicaise2019non}.

The presentation in \Cref{eq:Aisquotient} gives $A$ the structure of an abelian group, where the operation is that induced from multiplication in $(\Lambda^\x)^n$.
We denote by $\mu: A \x A \to A$ the product structure on $A$.
The induced map on Chow groups (c.f. \Cref{eq:CH_pontryagin})
\begin{align}\label{eq:pontryaginproductalg}
    \begin{split}
        \star:\CH_p(A) \otimes \CH_q(A) &\to \CH_{p+q}(A)\\
        Z_1 \star Z_2 &:= \mu_*(Z_1 \x Z_2)
    \end{split}
\end{align}
is called the \emph{Pontryagin product}.
It turns $\CH_*(A)$ into a ring, and $\CH_0(A) = \CH^n(A)$ is a subring for $\star$.
\begin{rmk}\label{rmk:compatiblegroupstructures}
    There is also a group structure $m:T^n \x T^n \to T^n$  on the base coming from the group structure of $\R^n$, and it is compatible with $\mu$ in the sense that $m(\log(f_1),\log(f_2)) = \log(\mu(f_1,f_2))$.
    Furtheremore, the inclusion
    \begin{align*}
        T^n &\to A \\
        (\lambda_1,\dots,\lambda_n) &\mapsto \left(T^{\lambda_1},\dots,T^{\lambda_n}\right)
    \end{align*}
    is a group homomorphism.
\end{rmk}

To make a connection with symplectic geometry let us now focus on the subring $\CH_0(A)$.
This should be mirror to the fibered cobordism group $\Cobfib(X)\subset \Cob(X)$ of the  mirror symplectic torus $X = X(B)$.
To define the analogue operation on the symplectic side,  we to consider the abelian structure $m:B \x B \to B$ of the base $B \simeq T^n$.

\begin{lem}\label{lem:pontryagin_on_cob}
    Let $B = T^n$ be the base of the Lagrangian torus fibration $\pi: X(B)\to B$.
    The operation on fibers
    \begin{align*}
        \star:\Z^B \x \Z^B &\to \Z^B\\
        F_p \star F_q &:= F_{m(p,q)}
    \end{align*}
    descends to a group homomorphism $\Cobfib(X) \otimes \Cobfib(X) \to \Cobfib(X)$.
\end{lem}
\begin{proof}
    Consider tuples of points $(p_i^\pm) \in B$ such that $\sum (F_{p_i^+} - F_{p_i^-}) = 0$ in $\Cobfib(X)$.
    We must show that 
    $$
        \sum (F_{m(p_i^+,q)} - F_{m(p_i^-,q)}) = 0
    $$
    in $\Cobfib(X)$.
    This follows from the observation that, if $t_b: T^{2n} \to T^{2n}$ is the symplectomorphism $(q,p) \mapsto (q+b,p)$, then  $F_{m(p_i^+,q)} = (t_q)_* F_{p_i^+}$. 
    Here, $(t_q)_*: \Cob(X) \to \Cob(X)$ denotes the induced map on cobordism groups (which is well-defined because $t_q$ is a symplectomorphism) which restricts to a map $(t_q)_*: \Cobfib(X) \to \Cobfib(X)$.
    One then has 
    $$
    \sum (F_{m(p_i^+,q)} - F_{m(p_i^-,q)}) = \sum (t_q)_* (F_{p_i^+} - F_{p_i^-}) = 0
    $$
    which proves the Lemma.
\end{proof}

\begin{defn}
    The homomorphism $\star: \Cobfib(X) \otimes \Cobfib(X) \to \Cobfib(X)$ is called the \emph{Pontryagin product}.
\end{defn}

\begin{rmk}
    The above operation can be enhanced to include gradings, local systems and all the data of \Cref{defn:lagcobgroup_thispaper}.
    For instance, for local systems $\xi_{b}, \xi_{b'}$ on  $F_b$ and $F_{b'}$, one defines 
    $$
        (F_b, \xi_b) \star (F_{b'}, \xi_{b'}) = (F_{m(b,b')}, (t_{b'})_* \xi_b \otimes (t_{b})_*\xi_{b'}).
    $$
    The reason we don't need to deal with this is \Cref{prop:exchangeoncob}: there we show that the Pontryagin product can be rephrased as fiberwise addition using a different Lagrangian torus fibration.
    This is a particular instance of geometric composition of Lagrangian correspondences, and the transformation of local systems, Spin structures and gradings under geometric composition has already been developed in the literature \cite{wehrheim2010quilted,subotic,wehrheim2015orientations}.
\end{rmk}

\begin{remark}
    The Pontryagin product on the algebraic side is defined on the whole Chow group (see \Cref{eq:pontryaginproductalg}), whereas in the symplectic side we have only given a definition for the fibered cobordism group.
    We show in  \Cref{sec:exchange} that the symplectic Pontryagin product is the same as the fiberwise addition operation when we equip $X(B)$ with the Lagrangian torus fibration $X(B) \to T_0^*T^n/T_{\Z,0}^*T^n$ given by projecting to the fibers.
    The operation of fiberwise addition was studied by Subotic in his thesis \cite{subotic}, where he showed it is a well-defined operation on some enlarged version of the Fukaya category consisting of generalized Lagrangian correspondences.
    To get an operation between non-enlarged Fukaya categories one must deal with geometric composition of Lagrangian correspondences---and it is often the case that this yields immersed Lagrangians.
    However, work of Fukaya \cite{fukaya2017unobstructed} shows that, even if immersed, geometric composition of unobstructued Lagrangian correspondences produces unobstructed (immersed) Lagrangians.\footnote{For an embedded compact Lagrangian $L$, unobstructedness means there exists a \emph{bounding cochain} $b \in CF^1(L,L)$ satisfying the Maurer-Cartan equation 
    $$
    \sum_k m^k(b,\dots,b) = 0.
    $$
    Here $\{m^k\}$ are the $A_\infty$-operations of the Fukaya category. 
    This yields a deformed $A_\infty$-algebra $(CF(V,V),\{m^k_b\}_{k\in\N})$---as defined e.g. in \cite{fukaya2009lagrangian}---with $m_b^0 = 0$.
    For the general case of immersed Lagrangians we refer the reader to \cite{akaho2010immersed}.}
    Thus, if one is willing to enlarge the definition of the Lagrangian cobordism group to include unobstructed immersed Lagrangians, then the Pontryagin product operation can be defined on the whole cobordism group.
\end{remark}

Let us recall that in \Cref{assumption:graded_spin,assumption:generators,assumption:cobordisms} we considered a setting where all Lagrangians are oriented and Spin, all generators of the cobordism group are embedded and weakly-exact, and all cobordisms are immersed.
As mentioned in the above remark, Fukaya has shown that, at least in the compact case, geometric composition of immersed unobstructed Lagrangians is again unobstructed \cite{fukaya2017unobstructed}.
Unobstructedness of  Lagrangian cobordisms (where one has to deal with non-compactness) has been been studied by Hicks \cite{hicks2019wall}.
We expect that the results of Fukaya can be modified to Lagrangian cobordisms, so that \emph{geometric composition of immersed unobstructed Lagrangian cobordisms is again unobstructed}.
This has not appeared in the literature and is not the purpose of this paper either, but we assume it to be true for the following result.

Namely, a stronger condition to unobstructedness is \emph{tautological unobstructedness} (also known as `quasi-exactness'): this means there exists an $\omega$-compatible almost complex structure such that the Lagrangian bounds no non-constant holomorphic disks (nor teardrops in the immersed case).
When the ambient symplectic manifold is \emph{aspherical} (i.e. $\omega(\pi_2(X))=0$, which is the case for our symplectic torus $X = X(B)$), Rathel-Fournier \cite[Theorem A]{rathel2023unobstructed} extended work of Biran-Cornea-Shelukhin \cite{biran2021lagrangian} and showed that tautologically unobstructed immersed cobordisms between weakly-exact embedded Lagrangians induce cone decompositions in the Fukaya category.
The generalisation to unobstructed---but not necessarily tautologically unobstructed---immersed cobordisms is expected to be true but has not yet appeared in the literature.
We will assume:
\begin{assumption}\label{assumption:decompositions}
    Unobstructed immersed cobordisms between weakly-exact embedded Lagrangians induce cone decompositions in the Fukaya category.
\end{assumption}

It follows from the assumption above that there is a well-defined map:
\begin{equation}\label{eq:Cob_to_K0}
    \Cob(X) \to K_0(\mcal Fuk(X))
\end{equation}
sending each Lagrangian $L \in \mcal L$ to its class $[L] \in K_0(\mcal Fuk(X(B)))$.
\begin{remark}
    \Cref{assumption:decompositions} is \emph{not} required for \Cref{thm:filtration_terminates}---the main result of this paper---nor for any of the results about the Fourier transform in \Cref{sec:symplecticfourier}.
    It is only needed to construct the map of \Cref{eq:Cob_to_K0} and to relate the cobordism group and the Chow group of the mirror as in the following Proposition.
\end{remark}
\begin{prop}\label{prop:Cob_to_CH_ringhom}
    Let $\Cobfib(X) = (\Cobfib(X), +, \star)$ be the fibered cobordism group of $X = X(B)$ equipped with the ring structure coming from Pontryagin product.
    The composite map
    $$
    \begin{tikzcd}
        K_0(\mcal Fuk(X)) \arrow[r,"\sim"] & K_0(D^bCoh(A)) \arrow[d,"ch"] \\
        \Cobfib(X) \arrow[u] \arrow[r,dashed,"\Psi"] & \CH^*(A)_\Q
    \end{tikzcd}
    $$
    induces a ring homomorphism $\Cobfib(X)  \to \CH_0(A)$.
\end{prop}
\begin{proof}
    First note that the Chern character maps skyscraper sheaves $\mcal O_p, p\in A$ to point classes $[p]\in \CH_0(A)$.
    This implies that, when restricted to the fibered cobordism group, the map $\Psi$ lands in the integral subgroup $\CH_0(A) \subset \CH_0(A)_\Q$.

    We now show the restriction  $\Psi: \Cobfib(X)  \to \CH_0(A)$ is a ring homomorphism.
    Let $F_{b_i} \subset X(B), i=1,2$ be two Lagrangian torus fibers and let $\mcal O_{p_i}, p_i \in A$ be the corresponding skyscraper sheaves on $A$.
    Note that for any $b \in B$ we have that  $\Psi(F_{b}) = [p]$, where $p \in A$ (or, more precisely, $\mcal O_p$) is the mirror to the Lagrangian torus fiber $F_b$.
    Then $\Psi:\Cobfib(X)  \to \CH_0(A)$ is a ring homomorphism if the point in $A$ mirror to $F_{b_1}\star F_{b_2}=F_{m(b_1,b_2)}$ is (rationally equivalent to) $\mu(p_1,p_2)$.
This is indeed the case, see \Cref{rmk:compatiblegroupstructures}.
\end{proof}

 \subsection{Fiberwise addition}\label{sec:fiberwiseaddition}

On any algebraic variety $Y$---in fact, any scheme---there is a tensor product operation turning $D^bCoh(Y)$ into a (symmetric) monoidal category. 
When $Y$ is furthermore smooth the Chow group $\CH^*(Y)$ is not just a group but a commutative graded ring, where the new operation on cycles 
$$
    \CH^i(Y) \x \CH^j(Y) \to \CH^{i+j}(Y)
$$
can be interpreted geometrically as an intersection \cite{fulton2013intersection}.
It is known that the Chern character  
$$
ch:K_0(D^bCoh(Y)) \to \CH^*(Y)_\Q
$$
intertwines these two operations.

Now suppose $X$ is homologically mirror to $Y$, so that there is an equivalence of categories $\mcal Fuk(X) \simeq D^bCoh(Y)$.
Then $\mcal Fuk(X)$ inherits formally a monoidal structure.
The geometry behind this new operation was studied by Subotic in his thesis \cite{subotic}.
Subotic showed that, in the setting of SYZ mirror symmetry---that is, in the presence of dual torus fibrations
$$
\begin{tikzcd}
    X \arrow[dr,"\pi"]& & Y \arrow[dl,"\pi^\vee"']\\
    & B
\end{tikzcd}
$$
---the tensor product operation corresponds to a `fiberwise addition' operation on $\mcal Fuk(X)$.
The precise definition of this operation requires the language of Lagrangian correspondences and their induced functors as developed by Wehrheim-Woodward and Ma'u-Wehrheim-Woodward in a series of papers \cite{wehrheim2009pseudoholomorphic, wehrheim2010functoriality,ma2018a_}.
The geometry behind this operation is however easy to understand: given two Lagrangians $L_1,L_2 \subset X$ satisfying suitable transversality conditions, their fiberwise addition $L_1\otimes L_2$ can be locally parametrized via
$$
L_1 \otimes L_2 := 
\{
    (q, p_1+ p_2) \in T^*\R^n / T^*_\Z \R^n \subset X \st q \in \pi(L_1) \cap \pi(L_2), p_i \in \pi\inv(q) \cap L_i
\}.
$$
More formally, the fiberwise addition of two Lagrangians is the geometric composition 
$$
    \mcal L_{\otimes} \circ (L_1 \x L_2)
$$
of the Lagrangian correspondences $L_1 \x L_2 \subset X \x X$ and the fiberwise addition correspondence \cite{subotic}
\begin{equation}\label{eq:fiberwiseaddition}
\mcal L_{\otimes} = \{(q, p_1, q, p_2, q, p_1+p_2) \mid q \in B, p_i \in T^*_q B\} \subset X \x X \x X.
\end{equation}

\begin{example}
    If $L_i = \Gamma(df_i)$ are the graphs of the differentials of functions $f_i: B \to \R$ (or, more generally, $f_i \in H^0(C^\infty/\Aff)$), then $L_1\otimes L_2 = \Gamma(d(f_1 +f_2))$.
    That is, the map
    $$
    \Gamma(d(-)): (H^0(C^\infty /\AffZ),+) \to (\mcal Ob(\mcal Fuk(X)),\otimes)
    $$
    is a group homomorphism (cf. \Cref{sec:tropicalgeometry}).
\end{example}

\subsection{Exchange of ring structures}\label{sec:exchange}
We saw in \Cref{prop:symplecticfourier} that the symplectic Fourier transform $\mathfrak F$ (the mirror to the Fourier transform $\mcal F$) is the equivalence 
$$
\mathfrak F = \varphi_\iota : \mcal Fuk(X(B)) \to \mcal Fuk(X(B^\vee))
$$ 
of \Cref{eq:gradediota} induced by the graded symplectomorphism $\iota:X(B) \to X(B^\vee)$ of \Cref{eq:iota}.
At the level of  cobordism groups, it induces the group homomorphism 
$$
L \mapsto \iota(L) =  \{(-p,q) \in X(B^\vee) \mid (q,p) \in X(B)\}.
$$

Some interesting geometry appears when we consider the two operations on Lagrangians described in \Cref{sec:pontryagin,sec:fiberwiseaddition}, namely Pontryagin product and fiberwise addition.
Recall that these are defined for a symplectic manifold \emph{plus} a choice of a Lagrangian torus fibration.
Given a Lagrangian torus fibration $\pi: X(B) \to B$, we denote the corresponding operations by $\otimes_\pi,\star_\pi$.
For such a fibration, we define the ring
$$
\Cob^{\dagger}(X(B)) \equiv (\Cob(X(B)), +, \dagger)
$$
where $\dagger \in \{\otimes_\pi,\star_\pi\}$.
\begin{prop}\label{prop:exchangeoncob}
    Let $\Cob_{\text{sec}}(X(B)) \subset \Cob(X(B))$ be the subgroup generated by Lagrangian sections of $X(B) \to B$.
    Let $\iota$ be the map defined in \Cref{eq:iota}.
    The group homomorphism
    \begin{align*}
        \mathfrak F_*:\Cobfib^{\star_\pi}(X(B)) &\to \Cob_{\text{sec}}^{\otimes_{\pi'}}(X(B^\vee)) \\
        F & \mapsto \iota(F)
    \end{align*}
    is an (injective) ring homomorphism.
\end{prop}
\begin{proof}
    We must prove that
    $$
    \iota(F_1) \otimes_{\pi'} \iota(F_2) = \iota(F_1 \star_\pi F_2).
    $$
    Now note that $\otimes_{\pi'}$ is $\iota$-equivariant, hence this reduces to showing that 
    $$
    \iota(F_1 \otimes_{\pi'} F_2) = \iota(F_1 \star_\pi F_2)
    $$
    i.e. $F_1 \otimes_{\pi'} F_2 = F_1 \star_\pi F_2$.
    One can check directly from \Cref{lem:pontryagin_on_cob} that $\star_\pi$ is induced from the same Lagrangian correspondence as $\otimes_{\pi'}$, completing the proof.
\end{proof}

\begin{rmk}
    The above Proposition can be used to prove \Cref{mainthm:filtration_terminates}. 
    Indeed, if 
    $$
    \mathbb L = (F_1^+ - F_1^-) \star_{\pi} \dots \star_{\pi} (F_{n+1}^+ - F_{n+1}^-)
    $$ is a generator of $F^{n+1} \Cobfib(X) = \Cobfib(X)_{\hom}^{\star n+1}$, it follows from the Proposition that $\mathbb{L} = 0$ if $\mathfrak F_*(\mathbb{L}) = 0$, which we can compute as
    \begin{align*}
        \mathfrak F_*(\mathbb{L}) &= \mathfrak F_*((F_1^+ - F_1^-) \star_{\pi} \dots \star_{\pi} (F_{n+1}^+ - F_{n+1}^-))\\
        & = \mathfrak F_*(F_1^+ - F_1^-) \otimes_{\pi'} \dots \otimes_{\pi'} \mathfrak F_*(F_{n+1}^+ - F_{n+1}^-) \qquad \tx{(\Cref{prop:exchangeoncob})}\\
        &= (\Gamma_1^+ - \Gamma_1^-) \otimes_{\pi'} \dots \otimes_{\pi'} (\Gamma_{n+1}^+ - \Gamma_{n+1}^-)
    \end{align*}
    where $\Gamma_i^\pm$ are  now flat sections of $X(B^\vee) \to B^\vee$.
    As in the proof of \Cref{thm:filtration_terminates}, we now apply \Cref{prop:fiberwise_sum_increases_codimension} to show that the above Lagrangian lives over tropical subvarieties of codimension at least $n+1$, hence vanishes.
    This is essentially the same proof as that of \Cref{mainthm:filtration_terminates}, but using the Fourier transform to exchange Pontryagin products for fiberwise sums.
\end{rmk}

\renewbibmacro{in:}{}
\def\bibrangedash{ -- }
\printbibliography

\end{document}